\documentclass[11pt,a4paper]{scrartcl}
\usepackage[top=2.5cm, bottom=2.5cm, left=2.5cm, right=2.5cm]{geometry}

\usepackage[ansinew]{inputenc}
\usepackage[T1]{fontenc}
\usepackage{subfig}
\usepackage{graphicx}

\usepackage{epstopdf}
\usepackage{xr}
\usepackage[format=plain,justification=centering,singlelinecheck=true]{caption}
\usepackage{float}
\usepackage{latexsym}
\usepackage{amsmath,amssymb,amsthm,tabu,amsfonts}
\usepackage{makeidx} 
\usepackage[numbers]{natbib}
\usepackage{enumitem}
\usepackage{pdfpages}
\usepackage{color}
\usepackage{todonotes}
\usepackage{colortbl}
\usepackage{xcolor}
\usepackage{pifont}
\usepackage{booktabs}
\usepackage[subfigure]{tocloft} 
\usepackage{hyperref}
\usepackage{caption}
\usepackage{adjustbox}
\usepackage[hang,flushmargin,multiple,ragged]{footmisc}

\usepackage{caption}
\allowdisplaybreaks

\newtheorem{theorem}{Theorem}
\newtheorem{definition}{Definition} 
\newtheorem{remark}{Remark}	   
\newtheorem{assumption}{Assumption}

\newtheorem{lemma}{Lemma}
\newtheorem{proposition}{Proposition}
\newtheorem{corollary}[theorem]{Corollary}
                  
\numberwithin{equation}{section}

\newcommand{\Var}{\operatorname{Var}}
\newcommand{\E}{\operatorname{E}}

\newcommand{\sign}{\operatorname{sign}}

\begin{document}
\title{Asymptotic delay times of sequential tests based on $U$-statistics for early and late change points}
\author{Claudia Kirch\footnote{\, Otto-von-Guericke University Magdeburg, Institute for Mathematical Stochastics; Center for Behavioral Brain Sciences (CBBS); Magdeburg, Germany; claudia.kirch@ovgu.de} \and Christina Stoehr\footnote{\, Ruhr-University Bochum, Faculty of Mathematics,  Bochum, Germany; christina.stoehr@ruhr-uni-bochum.de} }
\maketitle

\begin{center}
\begin{minipage}{0.8\textwidth}
\begin{center}\textbf{Abstract}\end{center}
Sequential change point tests aim at giving an alarm as soon as possible after a structural break occurs while controlling the asymptotic false alarm error. For such tests it is of particular importance to understand how quickly a break is detected. 
While this is often assessed by simulations only, in this paper, we derive the asymptotic distribution of the delay time for sequential change point procedures based on U-statistics. This includes the difference-of-means (DOM) sequential test, that has been discussed previously, but also a new robust Wilcoxon sequential change point test. Similar to asymptotic relative efficiency in an a-posteriori setting, the results allow us to compare the detection delay of the two procedures. It is shown that the Wilcoxon sequential procedure has a smaller detection delay for heavier tailed distributions which is also confirmed by simulations.  While the previous literature only derives results for early change points, we obtain the asymptotic distribution of the delay time for both early as well as late change points. Finally, we evaluate how well the asymptotic distribution approximates the actual stopping times for finite samples via a simulation study.
\end{minipage}\\
\end{center}


\section{Introduction}\label{sec.intro}
Monitoring time series for structural breaks have a long tradition in time series going back to \cite{Page1,Page2}.  
	In the seminal paper of	\cite{Chu} sequential change point tests are introduced that allow to control the asymptotic false alarm rate (type-I-error) while guaranteeing that the procedure will asymptotically stop with probability one if a change occurs by assuming the existence of a stationary historic data set. This approach has then been adopted  for a variety of settings and a variety of monitoring statistics. For example, \cite{aue2006change,Hor,huvskova2005monitoring} derive procedures in linear models, \cite{hlavka2012monitoring} for changes in the error distribution of	autoregressive time series,  \cite{gut2002truncated} for renewal processes while \cite{berkes2004sequential} consider changes in GARCH models. \cite{Est}  derive a unified theory based on estimating functions, that has been extended to different monitoring schemes by  \cite{kirch2018modified}. Bootstrap methods designed for the particular needs of sequential procedures have also been proposed by~\cite{hlavka2016bootstrap,huvskova2008bootstrapping,kirch2008bootstrapping}.
	In this paper, we will revisit the sequential test based on $U$-statistics that has been proposed in ~\cite{seqcp}, which includes a difference-of-means (DOM) as well as Wilcoxon sequential test. A similar setting for a-posteriori change point tests have been considered by~\cite{Csorg} for independent data, which has been extended to time series by \cite{Dehl}.

	Starting with \cite{aue2004delay}, who consider a mean change model,  several papers derived the limit distribution of the corresponding stopping times. For example \cite{aue2008monitoring,Fremdt} consider a modified test statistic for changes in the mean, \cite{aue2009delay,huvskova2009delay} consider changes in a linear regression models, while \cite{gut2009truncated} consider changes in renewal processes.
All of those papers only obtain results for early (sublinear) change points, where the time of change relative to the length of the historic data set vanishes asymptotically. In contrast, we will derive the corresponding results not only for such early change points but also for late (linear and even superlinear) relative to the length of the historic data set, which has been an open problem even for the mean change problem and the standard DOM statistic until now.

As in the setting of
\cite{Chu} we assume the existence of a stationary historic data set $X_1,\ldots, X_m$. Then, we monitor new incoming data by testing for a structural break after each new observation $X_{m+k},k\geq 1,$  based on a monitoring statistic $\Gamma(m,k)$. The procedure stops and detects a change as soon as $w(m,k)\left|\Gamma(m,k)\right|>c_m.$ The weight function $w(m,k)\ge 0$ is chosen such that $\sup_{k\ge 1}w(m,k)|\Gamma(m,k)|$ converges in distribution to some non-degenerate limit distribution (as the length of the historic data set $m$ grows to infinity) if no change occurs. If the threshold $c_m$ is chosen as the corresponding $(1-\alpha)$-quantile, the procedure has an asymptotic false alarm rate of $\alpha$ (while still having asymptotic power one under alternatives).

To illustrate, consider the mean change model
\begin{equation}\label{meanex}
X_{i,m}=Y_i+1_{\{i>k^*+m\}}d_m,\quad d_m\neq 0,
\end{equation}
where $\{Y_i\}_{i\geq 1}$ is a stationary time series with mean $\mu$. The change in the mean is given by $d_m$ and is allowed to depend on $m$.

For this situation the classical \textbf{difference-of-means (DOM) monitoring statistic} is given by
\begin{align}\label{cusumstat}
	\Gamma_D(m,k)=\frac{k}{m}\sum_{i=1}^{m}X_i-\sum_{j=m+1}^{m+k}X_j=\frac{1}{m}\sum_{i=1}^m\sum_{j=m+1}^{m+k}(X_i-X_j).
\end{align}
The corresponding sequential procedure has already been investigated by several authors including~\cite{aue2006change,Chu,Hor,huvskova2005monitoring}.
Clearly, this is a sequential version of a two-sample statistic that similarly to the two-sample $t$-test is not  robust. Given the good properties of a Wilcoxon/Mann-Whitney two-sample test it is promising to  consider the \textbf{Wilcoxon monitoring statistic}
\begin{align}\label{wilcoxstat}
	\Gamma_W(m,k)=\frac{1}{m}\sum_{i=1}^m\sum_{j=m+1}^{m+k}\left(1_{\{X_i<X_j\}}-1/2\right),
\end{align}
which was recently proposed by~\cite{seqcp}.
Both statistics are sequential $U$-statistics of the following type:
\begin{align}\label{detstat}
\Gamma(m,k)=\frac{1}{m}\sum_{i=1}^m\sum_{j=m+1}^{m+k}(h(X_i,X_j)-\theta),
\end{align}
where the kernel $h:\mathbb{R}^2\rightarrow\mathbb{R}$ is a measurable function and $\theta=\E(h(Y,Y_1))$ with $Y\stackrel{D}{=}Y_1$ is an independent copy of $Y_1$. 
In this framework, the DOM-kernel is given by $h_D(s,t)=s-t$ such that $\theta^D=\E\left(Y-Y_1\right)=0$. The Wilcoxon-kernel is given by $h_W(s,t)=1_{\{s<t\}}$, such that $\theta^W=P(Y<Y_1)=1/2$.

The stopping time of the corresponding sequential procedure is given by
	\begin{align}\label{defstop}
				\tau_m=\begin{cases}
				\inf\{k\geq 1: w(m,k)\,\left|\Gamma(m,k)\right|>c_{m}\},\\
				\infty,\quad \mbox{if } w(m,k)\,\left|\Gamma(m,k)\right|\leq c_{m}\mbox{ for all k},
				\end{cases}
				\end{align}
				where 
				$ w(m,k)^{-1}=m^{1/2}(1+k/m)$. The monitoring procedure stops as soon as the monitoring statistic $\Gamma(m,k)$ 
exceeds in absolute value a critical curve given by $c_m/w(m,k)$.
				If $c_m$ is chosen as the $(1-\alpha)$-quantile of the limit distribution as given in Theorem~\ref{as.H0}, the asymptotic false alarm rate is given by $\alpha$.  
				The above weight function is often chosen in the literature because the corresponding limit distribution as given in Theorem~\ref{as.H0} has a nice well-known structure (noting that the supremum is only taken over $(0,1)$ despite the infinite observation horizon).  \cite{seqcp} consider a much larger class of weight functions including in particular weight functions of the type
				\begin{align}\label{wgamma}
					w_{\gamma}(m,k)=m^{-1/2}\left(1+\frac km\right)^{-1}\left( \frac{k}{m+k} \right)^{-\gamma},\quad 0\le \gamma <\frac 1 2,
				\end{align}
				that also have this nice property. However, it is well known that a choice of $\gamma>0$ only results in a quicker detection for early changes (see also Remark~\ref{rem:gamma:early} below). Because the main focus of this paper lies on the analysis of the detection delay for late changes, where  the critical curve for larger values of $\gamma$ lies above the critical curve for $\gamma=0$ (see e.g. Figure 6.1 in \cite{diss}), the results are derived for $\gamma=0$ only. Nevertheless, for early changes the corresponding results for all $0\le \gamma <\frac 1 2$ have been considered in Section 5.1 in \cite{diss} and are summarized in Theorem~\ref{thm:delay:early} below.  
				\vspace{1mm}

This paper is organized as follows:
In Section~\ref{sec:as.delay} we derive the asymptotic delay times for $U$-statistics, where we first discuss the threshold selection in Section~\ref{sec_thre}.    In Section~\ref{sec_asym} we  obtain conditional results given the historic data, before we discuss corresponding consequences for the unconditional results in Section~\ref{sec:imp}. It turns out that for late changes, the influence of the historic data set to the asymptotic expected stopping time is no longer negligible. We explain how to obtain the standardizing sequences for the asymptotic results in Section~\ref{sec:derivation} giving a sketch of the main proof ideas along the way.
In Section~\ref{sec:sim} we  give some simulations indicating that the asymptotic result gives indeed a good approximation of the small sample behavior.
In Section~\ref{section_COM_Wil} we compare the DOM with the Wilcoxon procedure both based on theoretic considerations and by simulations.  After some conclusions in Section~\ref{section_conclusions} we finally provide the proofs in Section~\ref{section_proofs}.

\section{Asymptotic delay times}\label{sec:as.delay}
We consider the following model, which generalizes the above mean change model as given by \eqref{meanex},
\begin{equation}\label{meanmodel}
X_{i,m}=1_{\{1\leq i\leq k^*+m\}}Y_i+ 1_{\{i>k^*+m\}}Z_{i,m},\quad i\geq 1,
\end{equation}
where $\{Y_i\}_{i\in\mathbb{Z}}$ and $\{Z_{i,m}\}_{i\in\mathbb{Z}}$ are suitable stationary time series with unknown distribution, not necessarily centered fulfilling certain assumptions specified below. The distribution of the time series after the change and thus the change itself is allowed to depend on $m$. The change point fulfills $k^*=\lfloor \lambda m^{\beta}\rfloor$ with some unknown $\lambda>0$ and unknown $\beta\ge 0$.

The parameter $\beta$ effectively determines whether a change occurs early or late (compared to the length of the historic data set) and will influence the asymptotic distribution of the stopping time.
\begin{definition}
Changes $k^*=\lfloor \lambda m^{\beta}\rfloor$ with $0\leq \beta<1$ are classified as early or sublinear (in m) whereas changes with $\beta\geq 1$ are classified as late change. In the latter case we further distinguish between linear changes ($\beta=1$) and superlinear changes ($\beta>1$).
\end{definition}
In the previous literature, only the stopping times for early, i.e.\ sublinear, changes were obtained, while even for the DOM monitoring statistic the asymptotic distribution for late (i.e.\ linear and superlinear) changes has not been derived to the best of our knowledge.

Furthermore, define the (unknown) magnitude of the change $\Delta_m$ by
\begin{align}\label{eq_Delta}
	\Delta_m=\theta^*_m-\theta,\qquad
\theta^*_m=\E h(Y,Z_{1,m}),\qquad \theta=\E h(Y,Y_1),
\end{align}
where $Y\stackrel{D}{=}Y_1$ is independent of $Z_{1,m}$ and $Y_1$. We allow for fixed as well as local changes whose magnitude does decrease to 0, while at the same time being large enough to be asymptotically detectable with probability tending to one (see Remark~\ref{rem:test}):
 \begin{assumption}\label{change.ass}
$ $
\begin{itemize}
\item[(i)] $\Delta_m=O(1)$.
\item[(ii)] $\frac{\sqrt{m} |\Delta_m|}{c_m}\rightarrow \infty$.
\end{itemize}
\end{assumption}

For the DOM monitoring procedure with $h_D(s,t)=s-t$ we obtain by $\theta^D=\E\left( Y-Y_1 \right)=0$ 
 \begin{align}\label{delta.cusum}
	 \Delta^{D}_m=\theta^{*D}_m=\E(Y)-\E(Z_{1,m})=\E(Y_1)-\E(Z_{1,m}), 
 \end{align}
 such that in the mean change model~\eqref{meanex} it holds $\Delta^D_m=-d_m$.
 For the Wilcoxon monitoring procedure with $h_W(s,t)=1_{\{s<t\}}$ we obtain 
\begin{align}\label{delta.wil}
	&\Delta^{W}_m=P( Y<Z_{1,m})-P(Y<Y_1)=P( Y<Z_{1,m})-\frac 12.
\end{align}
such that in the mean change model~\eqref{meanex} it holds for $d_m>0$\begin{align*}
\Delta^{W}_m=
	P(Y_1\leq Y<Y_1+d_m)
\end{align*}
 and a similar expression for $d_m<0$.

 \subsection{Threshold selection}\label{sec_thre}
In this section, we discuss the selection of the threshold $c_m$.
Theorem~\ref{as.H0} states the limit distribution
of the monitoring statistic in the situation of no change (i.e. $k^*=\infty$) which allows to determine the threshold in such a way that the asymptotic false alarm rate is controlled by a previously chosen $\alpha$. Proposition~\ref{cinf} shows under which conditions the probability of a false alarm before a change occurs vanishes asymptotically. Most importantly, it shows that for late changes this can only be achieved if the threshold increases to infinity. This is in contrast to early changes as previously discussed in the literature.

In order to state the assumptions we need to introduce a version of Hoeffding's decomposition (see \cite{Hoeff}) which is widely used in the context of $U$-statistics. 
However, in contrast to the classical two-sample situations, the monitoring sample $\{X_{j,m}:m<j\le m+k\}$ contains both random variables following the distribution of $Y_i$ and those following $Z_{i,m}$ as soon as $m+k>k^*>0$. Therefore, additional terms appear taking this into account. Define
\begin{align}
	&h_1(s)=\E(h(s,Y_1)-\theta),\quad h^*_{1,m}(s)=\E(h(s,Z_{1,m})-\theta_m^*)\notag\\
	&h_2(t)=\E(h(Y_1,t)-\theta),\quad h^*_{2,m}(t)=\E(h(Y_1,t)-\theta_m^*),
\notag\\
&r(s,t)=h(s,t)-h_{1}(s)-h_{2}(t)-\theta,\quad r_m^*(s,t)=h(s,t)-h^*_{1,m}(s)-h^*_{2,m}(t)-\theta_m^*.\label{eq_hoeff}\end{align}
The terms $h_{1/2}$ and $r$ correspond to the usual Hoeffding's decomposition where both samples share the same distribution, while the terms $h_{1/2,m}^*$ and $r_m^*$ correpond to the ones where the second sample follows the distribution of $Z_{1,m}$.
Then, the following decomposition holds for $k>k^*$:
\begin{align}\label{GammaH1}
	\Gamma(m,k)&=\frac{k^*}{\sqrt{m}}\,S_{1,m}+\frac{k-k^*}{\sqrt{m}}\,S_{1,m}^*+\sqrt{k^*}\,S_{2,m}+\sqrt{k-k^*}\,S^*_{2,m}(k)+(k-k^*)\Delta_m+R_m(k)\notag\\
	&=:\widetilde{\Gamma}(m,k)+R_m(k).
\end{align}
Here, the signal part is given by $(k-k^*)\Delta_m$
as defined in \eqref{eq_Delta}, the first two summands are the historic parts involving
\begin{align*}
	S_{1,m}=\frac{1}{\sqrt{m}}\sum_{i=1}^m h_1(Y_i),\qquad S_{1,m}^*=\frac{1}{\sqrt{m}}\sum_{i=1}^m h^*_{1,m}(Y_i),
\end{align*}
the third and fourth summand are the monitoring part with 
\begin{align*}
	S_{2,m}=\frac{1}{\sqrt{k^*}}\sum_{j=m+1}^{m+k^*}h_2(Y_j),\qquad S^*_{2,m}(k)=\frac{1}{\sqrt{k-k^*}}\sum_{j=m+k^*+1}^{m+k}h^*_{2,m}(Z_{j,m}),
\end{align*}
while the last summand is a remainder term
\begin{align*}
&R_m(k)=\frac{1}{m}\sum_{i=1}^m\sum_{j=m+1}^{m+k^*}r(Y_i,Y_j)+\frac{1}{m}\sum_{i=1}^m\sum_{j=m+k^*+1}^{m+k}r_m^*(Y_i,Z_{j,m}).
\end{align*}
For $k\le k^*$ an analogous decomposition holds, where $k^*$ has to be replaced by $k$ making the terms involving $h_{1/2,m}^*$ and $r_m^*$ and $\Delta_m$ disappear (see also (2.1) in \cite{seqcp}), i.e.
\begin{align}\label{GammaH0}
\Gamma(m,k)=\sum_{j=m+1}^{m+k}h_2(Y_j)+\frac{k}{m}\sum_{i=1}^{m}h_1(Y_i)+\frac{1}{m}\sum_{i=1}^m\sum_{j=m+1}^{m+k}r(Y_i,Y_j).
\end{align}
For the DOM kernel with $h_D(s,t)=s-t$ we obtain by \eqref{delta.cusum}
\begin{align*}
&h^D_1(s)=h^{*D}_{1,m}(s)=s-E(Y_1),\quad h^D_2(t)=\E(Y_1)-t,\quad
 h^{*D}_{2,m}(t)=\E(Z_{1,m})-t,\\
&r^D(s,t)=r^{*D}(s,t)=0,
\end{align*}
where in the mean change model \eqref{meanex} it holds $\E(Z_{1,m})=\E(Y_1)+d_m$.

For the Wilcoxon kernel with $h_W(s,t)=1_{\{s<t\}}$ and a continuous $Y_1$ with distribution function $F$ it holds with \eqref{delta.wil}
\begin{align*}
&h^W_1(s)=\frac 12-F(s),\quad h^W_2(t)=F(t)-\frac 12,\quad r^W(s,t)=1_{\{s<t\}}+F(s)-F(t)-\frac 12\\
&h_{1,m}^{*W}(s)=\frac 12-P(Z_{1,m}\le s)-\Delta_m^W,\quad
h_{2,m}^{*W}(t)=F(t)-\Delta_m^W+\frac 12,\quad\\
&r^{*W}(s,t)=1_{\{s<t\}}+P(Z_{1,m}\le s)-F(t)+\Delta_m^W-\frac 32,
\end{align*}
where in the mean change model \eqref{meanex} it holds $P(Z_{1,m}\le s)=F(s-d_m)$.


\begin{assumption}\label{regass}
	Let $\{Y_i\}_{i\in\mathbb{Z}}$ be a stationary time series that fulfills the following assumptions for a given kernel function $h$ with the notation as in \eqref{eq_hoeff}:
	\begin{enumerate}
\item[(i)]
$\E\left(\left|\sum_{i=1}^m\sum_{j=k_1}^{k_2}r(Y_i,Y_j)\right|^2\right)\leq u(m)(k_2-k_1+1)\quad\mbox{for all } m+1\leq k_1\leq k_2$\\
with  $\frac{u(m)}{m^2}\log(m)^2\rightarrow 0$.
\item[(ii)]It holds for $k_m\to\infty$
	\begin{align*}
		\sup_{1\le k\le k_m}\frac{1}{\sqrt{k_m}}\left|\sum_{j=m+1}^{m+k}h_2(Y_j)\right|=O_P(1).
	\end{align*}
\end{enumerate}
\end{assumption}
The second assumption follows for example from a functional central limit theorem as in the next theorem. For a discussion of these assumptions for independent as well as dependent observations, we refer to Section~2.3 in \cite{seqcp}.

The proof of the following theorem can be found in \cite{seqcp} (Theorem 1 and Corollary 2), where the additional H\'ajek-R\'enyi-type inequality in that paper (as in \eqref{eq_hr_neu} below)  is only required for weight functions $w_{\gamma}$ as in \eqref{wgamma} with $\gamma>0$ but not for $\gamma=0$ as here.
\begin{theorem}
		\label{as.H0}
		Let Assumption \ref{regass} (i) be fulfilled in addition to the following functional central limit theorem (for any $T>0$)
	 	\begin{align*}
		\left\{\frac{1}{\sqrt{m}}\sum_{i=1}^{\lfloor m t\rfloor}\left(h_1(Y_i),h_2(Y_i)\right):0\le t\leq T\right\}\stackrel{D}{\rightarrow}\left\{\left({W}_1(t),{W}_2(t)\right):0\le t\leq T\right\},
	\end{align*}
where $\left\{\left({W}_1(t),{W}_2(t)\right):0< t\leq T\right\}$ is a non-degenerate centered bivariate Wiener process with $\sigma:=\Var({W_1}(1))=\Var({W_2}(1))$. 
Additionally, let the following H\'ajek-R\'enyi-type inequality hold: For any sequence $k_m>0$ it holds uniformly in $m$
	\begin{align*}
	\sup_{k\geq k_m}\frac{1}{k}\left|\sum_{j=1}^{k}h_2(Y_j)\right|=O_P\left(\frac{1}{\sqrt{k_m}}\right)\quad\mbox{as }k_m\rightarrow\infty.
	\end{align*}
	 Then, if no change occurs (i.e. $k^*=\infty$), it holds as $m\rightarrow\infty$
$$\sup_{k\geq 1}w(m,k)\left|\Gamma(m,k)\right|\stackrel{\mathcal{D}}{\rightarrow}\sigma \,\sup_{0<t<1}\left|W(t)\right|,$$
where $\{W(t):t>0\}$ is a standard Wiener processes.
\end{theorem}
If the assumption $\Var({W_1}(1))=\Var({W_2}(1))$ is not fulfilled, the limit distribution is more complicated (see Theorem~1 in \cite{seqcp}). However, both  the DOM and the Wilcoxon-kernel fulfill this assumption.

\begin{remark}\label{rem:test}
	The theorem shows that the monitoring procedure has asymptotic size $\alpha$ if the threshold $c_m$ is chosen as the $(1-\alpha)$-quantile $q_{1-\alpha}$ of the limit distribution in Theorem \ref{as.H0}. Furthermore, Theorem 3 and the respective proof in \cite{seqcp} imply that changes are detected asymptotically with probability one as $\sup_{k\geq 1}w(m,k)\left|\Gamma(m,k)\right|\stackrel{\mathcal{P}}{\rightarrow}\infty$ under Assumption \ref{change.ass} (ii); for more details we refer to \cite{seqcp}.
\end{remark}
Under the following assumptions the probability of a false alarm before the change occurs converges to zero.  This is necessary in order to have stopping times that are asymptotically not contaminated by false alarms. 
\begin{assumption}\label{ass_cm}
Let the following assumptions on the threshold $c_m$ hold:
\begin{align}
	&\liminf_{m\to \infty} c_m>0,
	\label{eq_lower_cm}\\
\label{cms}
&	P\left(\frac{|S_{1,m}|}{c_m}\frac{\lambda}{m^{1-\beta}+\lambda}<1\right) \to 1.
\end{align}
where $S_{1,m}:=\frac{1}{\sqrt{m}}\sum_{i=1}^m h_1(Y_i)$ and $\lambda$ and $\beta$ are defined by $k^*=\lfloor \lambda m^{\beta}\rfloor$. For late changes let $c_m\to\infty$. \end{assumption}
For a bounded sequence of critical values $c_m=O(1)$ assertion \eqref{cms} cannot be fulfilled for late changes if $S_{1,m}$ fulfills a central limit theorem  (an assumption that is required to derive the limit distribution in the no-change situation).
On the other hand,
for $S_{1,m}=O_P(1)$ assertion \eqref{cms} is automatically fulfilled as soon as $c_m\max(1,m^{1-\beta})\to\infty$, which holds for any sequence fulfilling \eqref{eq_lower_cm}  (including using asymptotic quantiles as critical values) for early changes, and for any sequence $c_m\to\infty$ for late changes.

\begin{proposition}\label{cinf}
Under Assumptions \ref{regass} and \eqref{ass_cm} it holds
\begin{align}\label{ph0}
P\left(\sup_{1\leq k\leq k^*}\frac{\left|\Gamma(m,k)\right|}{\sqrt{m}\left(1+\frac km\right)}>c_m\right)\rightarrow 0.
\end{align}
\end{proposition}

 \subsection{Asymptotic distribution of the delay times}\label{sec_asym}
In this section we consider the 
following delay time
\begin{align}\label{tautildedef}
{\kappa}_m:=\inf\left\{k>k^*:w(m,k)\left|\Gamma(m,k)\right|> c_m\right\},
\end{align}
which in contrast to the stopping time as in \eqref{defstop} explicitely excludes  early rejections due to false alarms. 
By choosing a threshold satisfying Assumption \ref{ass_cm}, early rejections are prevented asymptotically with probability one such that the delay time ${\kappa}_m$ is asymptotically equivalent to the stopping time $\tau_m$. 

Effectively, there are two essential differences when considering late changes as opposed to early changes:
\begin{itemize}
\item As stated in Proposition~\ref{cinf} for late changes the threshold $c_m$ needs to diverge in order to asymptotically guarantee that no false alarm prior to the change has occurred. For early changes fixed thresholds (such as quantiles of the  limit distribution in the no change situation) can and have been used.
\item For early changes, the influence of the historic data set on the stopping time is asymptotically negligible. This is no longer the case for late changes. Heuristically, this can be seen from decomposition~\eqref{GammaH1} where $S_{j,m}$, $j=1,2$, and $S_{j,m}^*$, $j=1,2$, are of the same order, so that the factors effectively determine the dominating terms. For early changes only $o(m)$ observations are necessary to reliably detect a change point (see Lemma~\ref{det.ass.gen} b)), hence the factors in front of $S_{1,m}$ and $S_{1,m}^*$ are of smaller order than the ones in front of $S_{2,m}$ and $S_{2,m}^*$, making these terms asymptotically negligible (for a mathematically rigorous proof of this statement we refer to Lemma 5.6 in \cite{diss}). On the other hand, for late changes, this is no longer true, in fact for linear changes with $\beta=1$ all four terms are of the same order while for late changes the terms involving the historic data set even dominate.
\end{itemize}
A consequence of the second observation is that for late change points the asymptotic stopping time can no longer be expected to be independent of the historic data set or more precisely of $S_{1,m}$ and $S_{1,m}^*$. In order to deal with this, we will first derive results conditional on $S_{1,m}$ and $S_{1,m}^*$ that will then also lead to unconditional results as well. For this reason we assume that the monitoring data set is independent of the historic data set so that $S_{2,m}$ and $S_{2,m}^*$ do not depend on the conditioning variables $S_{1,m}$ and $S_{1,m}^*$. 
\begin{assumption}\label{ass_in}
	Let the monitoring data set $\{X_{m+j}: j\ge 1\}$ be independent of $S_{1,m}$ and $S_{1,m}^*$.
\end{assumption}
This assumption is only required to get (for any index set $I$ and $z_m$)
\begin{align}\label{eq_condprob}
	&P\left(\sup_{k\in I}w(m,k)\left|\widetilde{\Gamma}(m,k)\right|\le z_m\,\Big|\,S_{1,m}=s_{1,m},S^*_{1,m}=s^*_{1,m}\right)\notag\\
	&=
	P\left( \sup_{k\in I}w(m,k)\left|\frac{k^*}{\sqrt{m}}\,s_{1,m}+\frac{k-k^*}{\sqrt{m}}\,s_{1,m}^*\right.\right.\notag\\
	&\phantom{=P(\sup_{k\in I}w(m,k)|\,\,} \left.\left.+\sqrt{k^*}\,S_{2,m}(k)+\sqrt{k-k^*}\,S^*_{2,m}(k)+(k-k^*)\Delta_m \right|\le z_m\right),
\end{align}
which clearly follows from the independence assumption.
As soon as this equality holds at least asymptotically, the independence assumption can be dropped. 
	We conjecture that this is possible if a suitable dependence structure allowing for example for big-block-small-block arguments is being used. However, for clarity of presentation of the results for monitoring $U$-statistics, we will leave this for future work.

\begin{assumption}\label{ass.stop}
	\begin{enumerate}
		\item[(a)]
Let $\{Y_i\}_{i\in\mathbb{Z}}$ and $\{Z_{i,m}\}_{i\in\mathbb{Z}}$ be stationary time series that fulfill the following assumptions for a given kernel function $h$:
\begin{itemize}
	\item[(i)]$\E\left(\left|\sum_{i=1}^m\sum_{j=k_1}^{k_2}r_m^*(Y_i,Z_{j,m})\right|^2\right)\leq u(m)(k_2-k_1+1)\quad\mbox{for all } m+1\leq k_1\leq k_2$\\
with  $\frac{u(m)}{m^2}\log(m)^2\rightarrow 0$.
\item[(ii)]
For all $0\leq\alpha<\frac{1}{2}$ the following Hajek-Renyi-type inequality holds
\begin{equation*}
\sup_{1\leq l\leq l_m}\frac{1}{m^{\frac{1}{2}-\alpha}l^{\alpha}}\left|\sum_{j=m+k^*+1}^{m+k^*+l}h^*_{2,m}(Z_{j,m})\right|=O_P\left(1\right)\quad\mbox{as }l_m\rightarrow\infty.
\end{equation*}
\end{itemize}
\item[(b)] For very early changes with 
	$\lim_{m\to\infty} m^{\beta -1}\,\frac{\sqrt{m}|\Delta_m|}{c_m}<\infty$ (which is a certain subclass of early sublinear changes),
	 we require additionally to (a) 
	the following functional central limit theorem 
	\begin{align*}
\left\{\frac{1}{\sqrt{k_m}}\sum_{j=1}^{[k_mt]}(h_2(Y_{j}),h^*_{2,m}(Z_{j,m})):0< t\leq 1\right\}\stackrel{D}{\rightarrow}\left\{\left(W(t),W^*(t)\right):0< t\leq 1\right\},
\end{align*}
for $k_m\rightarrow\infty$,
where $\left\{\left(W(t),W^*(t)\right):0< t\leq 1\right\}$ is a non-degenerate centered bivariate Wiener process with $\sigma^2=\Var(W(1))$ and $\sigma^{*2}=\Var(W^*(1))$.
\item[(c)] For later changes with $\lim_{m\to\infty} m^{\beta -1}\,\frac{\sqrt{m}|\Delta_m|}{c_m}\to \infty$, 
 we require additionally to (a)
\begin{itemize}
\item[(i)]
	$\frac{1}{\sqrt{k_m}}\sum_{j=1}^{k_m}h_2(Y_{j})\stackrel{\mathcal{D}}{\rightarrow}W(1)$ as $k_m\rightarrow\infty$, where $\{W(\cdot)\}$ is as in (b).
\item[(ii)] $\sup_{1\leq l\leq k_m}\frac{1}{\sqrt{k_m}}\sum_{j=m+k^*+1}^{m+k^*+l}
h_{2,m}^*(Z_{j,m})=O_P(1)$ as $k_m\rightarrow\infty$.
\end{itemize}
\end{enumerate}
\end{assumption}

%

The variance of the limit distribution of the delay time for early changes depends on the interplay between the magnitude of the change in combination how early the change occurs. More precisely with $k^*=\lfloor \lambda m^{\beta}\rfloor$ and the notation of Assumption~\ref{ass.stop}, consider
\begin{align}\label{def:sigmalim}
	\sigma_{\infty}^2=\begin{cases}
		\sigma^{*2}, &\phantom{\text{if }m^{\beta -1}\,\frac{\sqrt{m}|\Delta_m|}{c_m}
	\to \;} 0,\\
	\frac{\lambda D}{1+\lambda D}\,\sigma^2+\frac{1}{1+\lambda D}\,\sigma^{*2}
	,&\text{if }	m^{\beta -1}\,\frac{\sqrt{m}|\Delta_m|}{c_m}\to D\in (0,\infty),\\
		\sigma^2,&\phantom{\text{if }m^{\beta -1}\,\frac{\sqrt{m}|\Delta_m|}{c_m}\to \;}\infty.
\end{cases}
\end{align}
By Assumption~\ref{change.ass} (ii) for late changes always the last case that applies.
The assumptions on the threshold are clearly fulfilled for any constant threshold for early changes as well as for any sequence of thresholds converging to infinity in case of late changes.

The following assumptions imply  Assumption~\ref{ass_cm}.
\begin{assumption}\label{ass_neu_stoch}
	Let $S_{1,m}=O_P(1)$, $S_{1,m}^*=O_P(1)$, $\liminf_{m\to\infty}c_m>0$ as well as $$c_m\max(1,m^{1-\beta})\to\infty.$$
\end{assumption}

The below theorem shows that properly standardized delay times will be asymptotically normal with the above variance $\sigma_{\infty}^2$.  This shows that for sublinear changes with a combination of a small $\beta$ (particularly early) and/or small magnitude $\Delta_m$ the asymptotic variance of the delay time is determined by distributional properties of the observations after the change occurred. On the other hand for late (linear and superlinear) changes or even sublinear changes but with a larger magnitude of the change, the asymptotic variance is determined by distributional properties of the observations before the change occurs. The second case is the transition between those two situations.

\begin{theorem}\label{thm:cond}
	Let Assumptions ~\ref{change.ass} -- \ref{ass_cm} be satisfied. 
		Consider the standardizing sequences 
\begin{align}\label{defam}
&a_m(S_{1,m},S^*_{1,m})\\
&=\left(1-\frac{c_m}{\sqrt{m}|\Delta_m|}+\frac{S_{1,m}^*\sign(\Delta_m)}{\sqrt{m}|\Delta_m|}\right)^{-1}\left(k^*+k^*\frac{(S_{1,m}^*-S_{1,m})\sign(\Delta_m)}{\sqrt{m}|\Delta_m|}+\frac{c_m\sqrt{m}}{|\Delta_m|}\right),\notag
\end{align}
where $\sign(\Delta_m)$ is the sign of $\Delta_m$,
and
\begin{align}\label{defbm}
b_m(S_{1,m},S^*_{1,m})=\frac{ \sqrt{a_m(S_{1,m},S^*_{1,m})}}{|\Delta_m|}.
\end{align}

		Then, it holds
		$$P\left(\left.\frac{\kappa_m-a_m(S_{1,m},S^*_{1,m})}{b_m(S_{1,m},S^*_{1,m})}\leq x\right|S_{1,m},S^*_{1,m}\right)\overset{P}{\rightarrow}\Phi\left(\frac{x}{\sigma_{\infty}}\right),$$
		where $\Phi$ is the distribution function of the standard Gaussian distribution and $\sigma_{\infty}$ is as in \eqref{def:sigmalim}.
\end{theorem}

In particular the theorem shows that the sequence $a_m$ represents the asymptotic expected delay while $b_m$ represents the asymptotic variability.

\begin{assumption}\label{ass_as_S}
Let
\begin{align}
&	S_{1,m}=o\left( \sqrt{m}|\Delta_m| \right)\quad a.s., \qquad S^*_{1,m}=o\left( \sqrt{m}|\Delta_m| \right)\quad a.s. \notag\\
&	\limsup_{m\rightarrow\infty}\frac{|S_{1,m}|}{c_m}\frac{\lambda}{m^{1-\beta}+\lambda}<1 \quad\mbox{a.s. }\label{eq_S_lil}.
\end{align}
\end{assumption}
If a law of iterated logarithm for $S_{1,m}$ holds we can quantify $c_m$ for the latter condition (that also implies \eqref{cms}) precisely. 
\begin{remark}\label{rem_cm}
	Often $\{h_1(Y_j)\}$ fulfills a law of iterated logarithm, i.e.\
\begin{align}\label{lil}
&\limsup_{m\rightarrow\infty}\frac{\left|S_{1,m}\right|}{\sqrt{2\varrho^2\log\log m}}=1\quad\mbox{a.s.}
\end{align}
for some $\varrho^2>0$. 
In case of the DOM kernel this comes down to the usual law of iterated logarithm for $\{Y_j\}$, while for the Wilcoxon kernel it comes down to a law of iterated logarithm for $\{F(Y_j)\}$. In this case \eqref{eq_S_lil} is fulfilled for \textbf{early} changes (i.e.\ with $\beta<1$) if a constant threshold (for example obtained as a quantile from the limit distribution in Theorem~\ref{as.H0}) is applied. 

On the other hand, for \textbf{late} changes ($\beta\ge 1$) assertion \eqref{eq_S_lil} is not satisfied for a constant threshold but is fulfilled as soon as 
	\begin{align}\label{eq_neu_rem_2}\liminf_{m\to\infty}\frac{c_m^2}{\log\log m}>2\varrho^2,\end{align}
i.e.\ in particular, if $c_m$ grows faster to infinity than $\sqrt{\log\log m}$. In case of linear changes with $\beta=0$, this condition can be weakened further, where the left hand side only needs to be larger than $2\rho^2 \lambda^2/(1+\lambda)^2$.

Similarly, by the law of iterated logarithm the other assertions can also be quantified.
 \end{remark}

\begin{corollary}\label{cor:cond}
	Let Assumptions~\ref{change.ass} -- \ref{ass.stop} as well as \ref{ass_as_S} hold.
	Then, we get with the notation of Theorem~\ref{thm:cond}
	$$P\left(\left.\frac{\widetilde{\kappa}_m-a_m(S_{1,m},S^*_{1,m})}{b_m(S_{1,m},S^*_{1,m})}\leq x\right|S_{1,m},S^*_{1,m}\right){\rightarrow}\, \Phi\left(\frac{x}{\sigma_{\infty}}\right)\qquad a.s.,$$
	$\widetilde{\kappa}_m$ is defined as in  \eqref{tautildedef} but replacing $\Gamma(m,k)$ with $\widetilde{\Gamma}(m,k)$ as in \eqref{GammaH1}.
\end{corollary}
This is of particular interest for the DOM kernel, 
where the remainder terms are equal to zero such that $\kappa$ and $\widetilde{\kappa}$ coincide.

Theorem and corollary follow relatively easily from the below result  for the conditional probability given $S_{1,m}=s_{1,m}$ and $S^*_{1,m}=s^*_{1,m}$. This may also be of independent interest as the historic data is known when the monitoring starts (but $h_1$ and $h_{1,m}^*$ will typically depend on unknown distributional properties of the observations).  

As in Corollary~\ref{cor:cond} we only get the result for $\widetilde{\kappa}_m$. 
 Because the remainder term depends on the historic data set as well, its neglibility conditionally on $S_{1,m}=s_{1,m}$ and  $S^*_{1,m}=s^*_{1,m}$ cannot be established based on Assumptions \ref{regass} (i) and \ref{ass.stop} (a) (i) alone. On the other hand to prove the neglibility conditionally on $S_{1,m}$ and $S_{1,m}^*$, these assumptions are sufficient (see the proof of Theorem~\ref{thm:cond}).

 To obtain the below result, we need the sequences $s_{1,m}$ and $s^*_{1,m}$ to fulfill \eqref{eq_s_m}. By Assumption~\ref{ass_as_S} the corresponding sequences of random variables $S_{1,m}$ and $S_{1,m}^*$  fulfill \eqref{eq_s_m} almost surely. 
 Despite the fact that conditional probabilities are only defined almost surely, this does not mean that we can drop \eqref{eq_s_m} and still obtain the same limit distribution. This is because the  'almost surely' for the conditional probability refers to a one-set for each fixed $m$, whereas 'almost surely' as in Assumption~\ref{ass_as_S}  refers to a one-set with respect to a limit result (which is not uniform in $\omega$).
\begin{theorem}\label{thm:cond:version}
	Let  $s_{1,m}, s_{1,m}^*\in\mathbb{R}$ be sequences fulfilling 
	\begin{align}\label{eq_s_m}
		\limsup_{m\to\infty}\frac{|s_{1,m}|}{c_m}\,\frac{\lambda}{m^{1-\beta}+\lambda}<1,\qquad \frac{|s_{1,m}|}{\sqrt{m}|\Delta_m|}\to 0,\qquad \frac{|s_{1,m}^*|}{\sqrt{m}|\Delta_m|}=O(1).
	\end{align}
	Under the assumptions of Theorem \ref{thm:cond} it holds
	$$\lim_{m\rightarrow\infty}P\left(\left.\frac{\widetilde{\kappa}_m-a_m(s_{1,m},s^*_{1,m})}{b_m(s_{1,m},s^*_{1,m})}\leq x\right|S_{1,m}=s_{1,m},S^*_{1,m}=s_{1,m}^*\right)=\Phi\left(\frac{x}{\sigma_{\infty}}\right)\quad\mbox{for all }x\in\mathbb{R},$$
	where $\widetilde{\kappa}_m$ is defined as in  \eqref{tautildedef} but replacing $\Gamma(m,k)$ with $\widetilde{\Gamma}(m,k)$ as in \eqref{GammaH1}.
\end{theorem}
The main ideas to obtain the standarizing sequence $a_m$ and $b_m$ as well as of the proof of the above theorem can be found in Section~\ref{sec:derivation}.

The assumptions on the sequences $\{s_{1,m}\}$ and $\{s^*_{1,m}\}$ are fulfilled $P$-stochastically by Assumption~\ref{ass_neu_stoch} resp.\ almost surely by Assumption~\ref{ass_as_S} which is important in the proof of Theorem~\ref{thm:cond} as well as Corollary~\ref{cor:cond}.

\subsection{Consequences for unconditional results}\label{sec:imp}
We obtain the following unconditional results as
 an immediate corollary of  Theorem \ref{thm:cond} by  an application of Lebesgue's dominated convergence theorem
 \begin{corollary}\label{cor_uncond}
	Under the assumptions of Theorem~\ref{thm:cond} it holds with the same notation
	\begin{align*}
\frac{\kappa_m-a_m(S_{1,m},S^*_{1,m})}{b_m(S_{1,m},S^*_{1,m})}\stackrel{D}{\rightarrow}N\left(0,\sigma_{\infty}^2\right).
\end{align*}
\end{corollary}

The normalizing sequences in this theorem clearly still depend on the historic data via
 $S_{1,m}$ and $S_{1,m}^*$. 
However, the following theorem shows that the influence on the asymptotic variance is negligible, while for the asymptotic delay time the influence of the historic data is only negligible for early changes.
For late but linear changes the dependence on $S_{1,m}^*$ is negligible, which is still true for superlinear changes if $\beta$ is not too large (i.e.\ the change not too late) in combination with the magnitude of the change $\Delta_m$ being sufficiently large.
While $S_{1,m}$ only depends on distributional properties of the random variables before the change,  $S_{1,m}^*$ also depends on distributional properties after the change.

In accordance with \eqref{defam} define
\begin{align}
&a_m(S_{1,m},0)=\left(1-\frac{c_m}{\sqrt{m}|\Delta_m|}\right)^{-1}\left(k^*-k^*\frac{S_{1,m}\sign(\Delta_m)}{\sqrt{m}|\Delta_m|}+\frac{c_m\sqrt{m}}{|\Delta_m|}\right),\notag\\
&a_m(0,0)=\left(1-\frac{c_m}{\sqrt{m}|\Delta_m|}\right)^{-1}\left(k^*+\frac{c_m\sqrt{m}}{|\Delta_m|}\right),\label{eq_am_00}
\end{align}
and for $b_m$ accordingly. 

 \begin{theorem}\label{thm:aprime}
	 Under the assumptions of Theorem~\ref{thm:cond} it holds with the same notation:
\begin{itemize}
	\item[a)] For any $\beta \ge 0$ it holds
		\begin{align*}
			\frac{\kappa_m-a_m(S_{1,m},S^*_{1,m})}{b_m(0,0)}\stackrel{D}{\rightarrow}N\left(0,\sigma_{\infty}^2\right).
		\end{align*}
\item[b)] Let 
	\begin{align*}
		\frac{c_m}{\sqrt{m}\,|\Delta_m|}\,\max(1,m^{\frac{\beta-1}{2}})\to 0,
	\end{align*}
	which holds for all sublinear and linear changes but also some superlinear (not too late with sufficiently large magnitude) changes.
Then, it holds that
\begin{align*}
	\frac{\kappa_m-a_m(S_{1,m},0)}{b_m(0,0)}\stackrel{D}{\rightarrow}N\left(0,\sigma_{\infty}^2\right)
\end{align*}
\item[c)]For early (sublinear) changes with $\beta<1$ it holds
	\begin{align*}
	\frac{\kappa_m-a_m(0,0)}{b_m(0,0)}\stackrel{D}{\rightarrow}N\left(0,\sigma_{\infty}^2\right)
\end{align*}
\end{itemize}
\end{theorem}

\begin{remark}\label{rem_1_imp}
	For early changes, the factor $\left(1-\frac{c_m}{\sqrt{m}|\Delta_m|}\right)^{-1}$ in $a_m(0,0)$ is negligible (see Section~\ref{proof_sec_imp} for a proof). 
	In the previous literature the asymptotic delay time was thus reported as $k^*+\frac{c_m\sqrt{m}}{|\Delta_m|}$
	in corresponding results. 
	However, even for early changes, where this term is negligible (in contrast to late changes), the term can be seen as a bias correction leading to considerably better approximations in small samples as can be seen in Figure  \ref{figure3_neu} below. This is due to the fact, that this term is always strictly greater than $1$ such that the  asymptotic distribution without the bias correction systematically underestimates the delay time in small samples.\end{remark}

For early changes, the assertion of Theorem \ref{thm:aprime} c) has also been established in Theorem 5.3 in \cite{diss} for different weight functions with $\gamma>0$ under a slightly different set of assumptions. 
\begin{theorem}\label{thm:delay:early}
	Let the change fulfill Assumption~\ref{change.ass} with $c_m$ constant. Additionally let the time series before the change fulfill the assumptions of Theorem 1, where  Assumption~\ref{regass}(i) needs to be strengthend to fulfill $\frac{u(m)}{m^{1+\delta}}\to 0$ for all $\delta>0$ in addition to 
	\begin{align}\label{eq_hr_neu}
\sup_{1\leq k\leq m}\frac{1}{m^{\frac{1}{2}-\alpha}k^{\alpha}}\left|\sum_{j=1}^{k}h_2(Y_j)\right|=O_P\left(1\right)\quad\mbox{ as }m\rightarrow\infty
\end{align}
	 for all $0\leq\alpha<\frac{1}{2}$.
	Additionally, let Assumption~\ref{ass.stop}(a) hold with the same stronger assumption on $u(m)$ as above, Assumption~\ref{ass.stop} (b) and $S_{1,m}^*=O_P(1)$.

Denote by $\kappa_m(\gamma)$ the delay time defined as in \eqref{tautildedef} with $w(m,k)$ replaced by $w_{\gamma}(m,k)$ as in \eqref{wgamma}. Then, it holds for all $\beta<1$ 
	$$\frac{\kappa_m(\gamma)-a_m(\gamma)}{b_m(\gamma)}\stackrel{P}{\rightarrow}N(0,\sigma_{\infty}^2(\gamma)),$$
with\begin{align*}
&\sigma_{\infty}^2(\gamma)=
\begin{cases}
		\sigma^{*2}, &\phantom{\text{if }m^{(\beta -1)(1-\gamma)}\,\frac{\sqrt{m}|\Delta_m|}{c_m}
	\to \;} 0,\\
\xi\,\sigma^{*2}+(1-\xi)\,\sigma^2,&\text{if }	m^{(\beta -1)(1-\gamma)}\,\frac{\sqrt{m}|\Delta_m|}{c_m}\to D\in (0,\infty),\\
		\sigma^2,&\phantom{\text{if }m^{(\beta -1)(1-\gamma)}\,\frac{\sqrt{m}|\Delta_m|}{c_m}\to \;}\infty,
	\end{cases}
\end{align*}
 $\sigma^2$ and ${\sigma^*}^2$ as before
and $\xi=1-\frac{\lambda^{\gamma-1}}{D}\xi^{1-\gamma}$ as well as
\begin{align*}
a_m(\gamma)=\left(\frac{c_mm^{\frac 12-\gamma}}{|\Delta_m|}+\frac{k^*}{a_m^{\gamma}(\gamma)}\right)^{\frac{1}{1-\gamma}},\quad\quad b_m(\gamma)=\sqrt{a_m(\gamma)}|\Delta_m|^{-1}\left(1-\gamma\left(1-\frac{k^*}{a_m(\gamma)}\right)\right)^{-1}.
\end{align*}
\end{theorem}

\begin{remark}\label{rem:gamma:early}	The previous result can be used to establish that for sublinear changes a $\gamma$ close to $\frac 1 2$ leads to an asymptotically smaller delay time and is thus preferable for early changes (see Section~\ref{proof_sec_imp} for a proof).  However, this comes at the cost of a higher probability of a false alarm at the very beginning of the monitoring period (see e.g.\ Figure~1 in \cite{seqcp}).
	
	A corresponding analysis for late changes by means of the asymptotic delay times is left for future work. However, given that the procedure stops as soon as the critical curve (as defined by $c_m\,w_{\gamma}(m,k)^{-1}$) is exceeded by the monitoring statistic $|\Gamma(m,k)|$ one can easily compare the critical curves instead. It turns out that -- for constant critical values $c_m=c_{\gamma}$ chosen as $1-\alpha$-quantiles of the limit distribution in the no-change scenario as in Remark~\ref{rem:test} --  critical curves corresponding to smaller $\gamma$ are above those with larger $\gamma$ first, but drop beneath them relatively quickly showing that larger values of $\gamma$ are better for early changes while smaller ones are better for late changes. In fact, the simulations in \cite{kirch2018modified} show that this happens relatively soon, leading to best results for the choice $\gamma=0$. Indeed the quotient of two weight functions as in \eqref{wgamma} with $1/2>\gamma_1>\gamma_2\ge 0$ is given by
	\begin{align*}
		\frac{c_{\gamma_1}}{c_{\gamma_2}}\,\frac{w_{\gamma_2}(m,k)}{w_{\gamma_1}(m,k)}=\frac{c_{\gamma_1}}{c_{\gamma_2}}\,\left( \frac{k}{m+k} \right)^{\gamma_1-\gamma_2}
	\end{align*}
	where $\frac{c_{\gamma_1}}{c_{\gamma_2}}>1$. Consequently, this quotient is first smaller than one, but will eventually be larger than one. In fact, looking more closely, one can easily achieve that this happens for the first time, when $k>\eta\, m$ for some suitable $\eta>0$. In particular, this guarantees, that $\gamma=0$ achieves best results in the superlinear case, which is one of the reasons we restrict the main discussion in this paper to this case.
\end{remark}

\subsection{Asymptotic expectation and variance of the delay times and main proof ideas}\label{sec:derivation}
In this section we shed some light at the derivation of the standardizing sequences $a_m$, representing the asymptotic expected delay, as well as $b_m$, representing its asymptotic variability, which is a crucial step in the proof of Theorem~\ref{thm:cond:version} and thus ultimately of Theorem~\ref{thm:cond}.

With the notation of Theorem~\ref{thm:cond:version} the aim it is to derive
normalizing sequences $a_m=a_m(s_{1,m},s^*_{1,m})$ and $b_m=b_m(s_{1,m},s^*_{1,m})$ such that the limit distribution of
$(\widetilde{\kappa}_m-a_m)/b_m$
can be established. Denote  $P^*(\cdot):=P(\cdot|S_{1,m}=s_{1,m},S^*_{1,m}=s_{1,m}^*)$ with $\{s_{1,m}\}$, $\{s^*_{1,m}\}$ as in Theorem~\ref{thm:cond:version}.

Denoting 
\begin{align}\label{defN}
N:=N(m,x)=xb_m+a_m
\end{align}
 the main idea following \cite{aue2004delay} is the duality between the standardized delay time and the monitoring statistic:
\begin{align}\label{limtaun}
	\lim_{m\rightarrow\infty}P^*\left(\frac{\widetilde{\kappa}_m-a_m}{b_m}\leq x\right)&=1-\lim_{m\rightarrow\infty}P^*\left(\widetilde{\kappa}_m>N\right)\notag\\
	&=1-\lim_{m\rightarrow\infty}P^*\left(\sup_{k^*< k\leq N}w(m,k)\left|\widetilde{\Gamma}(m,k)\right|\leq c_m\right).
\end{align}

For other weight functions
it can be necessary to choose a more complicated $N$ that fulfills \eqref{limtaun}. For
 the weight function in \eqref{wgamma} with $0<\gamma\leq 1/2$, for example,
\begin{align*}
N(m,x)=\left(a_m^{1-\gamma}+x\frac{b_m}{a_m^{\gamma}}(1-\gamma)\right)^{\frac{1}{1-\gamma}}.
\end{align*}
 has been used in the proof of Theorem \ref{thm:delay:early} (see (5.8) in \cite{diss}) as well as by \cite{aue2004delay} and \cite{Fremdt}. For $\gamma=0$ the two definitions coincide.

 To get the desired result, in view of \eqref{limtaun}, it is sufficient to find sequences $e_m=e_m(s_{1,m},s^*_{1,m})$ and $v_m=v_m(s_{1,m},s^*_{1,m})$ such that
\begin{equation}\label{eq:cev}
	\frac{\left(c_m-e_m\right)}{\sqrt{v_m}}\rightarrow \frac{-x}{\sigma_{\infty}} \quad\mbox{as }m\rightarrow\infty
\end{equation}
and
\begin{align}\label{eq:normalize}
	P^*\left(\frac{\sup_{k^*< k\leq N}w(m,k)\left|\widetilde{\Gamma}(m,k)\right|-e_m}{\sqrt{v_m}}
\leq z\right)\rightarrow\Phi(z)\quad\mbox{for all } z\in\mathbb{R}
\end{align}
as then it holds
\begin{align*}
	P^*\left(\frac{\widetilde{\kappa}_m-a_m}{b_m}\leq x\right)&=1-P^*\left(\frac{\sup_{k^*< k\leq N}w(m,k)\left|\widetilde{\Gamma}(m,k)\right|-e_m}{\sqrt{v_m}}\leq \frac{c_m-e_m}{\sqrt{v_m}}\right)\\
&\rightarrow 1-\Phi\left( \frac{-x}{\sigma_{\infty}} \right)=\Phi\left( \frac{x}{\sigma_{\infty}} \right).
\end{align*}
To this end one first looks for sequences $e_m$ and $v_m$ such that \eqref{eq:normalize} holds. They will naturally depend on $N$ and thus on $a_m$ and $b_m$, such that the final task is to find
 $a_m$ and $b_m$ fulfilling \eqref{eq:cev}.

 So, first, let us have a closer look at \eqref{eq:normalize}: 
 Clearly, $e_m$ captures the expectation and $v_m$ the variance of $\sup_{k^*< k\leq N}w(m,k)\left|\widetilde{\Gamma}(m,k)\right|$.

 In the situation of this paper, the supremum in \eqref{eq:normalize} is dominated by $k\approx N$ (for a rigorous statement see Lemma \ref{delta.stop2} below) such that $e_m$ and $v_m$ can be constructed with respect to $\left|\widetilde{\Gamma}(m,N)\right|/g(m,N)$. The mean correction $e_m$ has to capture the signal part 
 as well as the historic parts of $\tilde{\Gamma}(m,k)$ as in \eqref{GammaH1} (the latter at least for late changes as discussed at the beginning of Section~\ref{sec_asym}), 
 where one has to take into account whether the historic parts strengthen or counteract the change.

 These considerations and \eqref{eq_condprob} lead to the choice
\begin{align}\label{eq:em}
e_m=\frac{(N-k^*)|\Delta_m|+\frac{k^*}{\sqrt{m}}s_{1,m}\sign(\Delta_m)+\frac{N-k^*}{\sqrt{m}}s_{1,m}^*\sign(\Delta_m)}{\sqrt{m}\left(1+\frac Nm\right)}.
\end{align}
Then, the limit distribution is determined by the monitoring part which for $k=N$ fulfills a central limit theorem with $N$ summands and thus requires a scaling of $1/\sqrt{N}$. In combination with the weight function this leads to the choice
\begin{align}\label{eq:vm}
	v_m=Nw^2(m,N)\,\sigma_{\infty}^2.
\end{align}
Proposition~\ref{delta.lim2} shows that the above choices of $e_m$ and $v_m$ indeed fulfill \eqref{eq:normalize}. By \eqref{defN}, \eqref{eq:em} and \eqref{eq:vm} we get 
\begin{align*}
&\frac{c_m-e_m}{\sqrt{v_m}}\\
&=\frac{|\Delta_m|}{\sqrt{a_m+x b_m}}\left[\frac{-x}{\sigma_{\infty}}b_m\left(1-\frac{c_m}{\sqrt{m}|\Delta_m|}+\frac{s_{1,m}^*\sign(\Delta_m)}{\sqrt{m}|\Delta_m|}\right)\right.\\
&\quad\quad\left.-a_m\left(1-\frac{c_m}{\sqrt{m}|\Delta_m|}+\frac{s_{1,m}^*\sign(\Delta_m)}{\sqrt{m}|\Delta_m|}\right)+k^*\left(1+\frac{(s_{1,m}^*-s_{1,m})\sign(\Delta_m)}{\sqrt{m}|\Delta_m|}\right)+\frac{c_m\sqrt{m}}{|\Delta_m|}\right].
\end{align*}
Choosing as in \eqref{defam}
\begin{align}\label{defam2}
a_m=\left(1-\frac{c_m}{\sqrt{m}|\Delta_m|}+\frac{s_{1,m}^*\sign(\Delta_m)}{\sqrt{m}|\Delta_m|}\right)^{-1}\left(k^*\left(1+\frac{(s_{1,m}^*-s_{1,m})\sign(\Delta_m)}{\sqrt{m}|\Delta_m|}\right)+\frac{c_m\sqrt{m}}{|\Delta_m|}\right)
\end{align}
the terms on the last line cancel, such that  by \eqref{eq_s_m} 
\begin{align*}
	\frac{c_m-e_m}{\sqrt{v_m}}=\frac{|\Delta_m|\,b_m}{\sqrt{a_m+x b_m}}\left[\frac{-x}{\sigma_{\infty}} (1+o(1))\right].
\end{align*}
By Assumption \ref{change.ass}(ii) and \eqref{eq_s_m} it holds
\begin{align}\label{eq:am:approx}
a_m=\left(k^*+\frac{c_m\,\sqrt{m}}{|\Delta_m|}\right)\, \left( 1+o(1) \right).
\end{align}
As $k^*\geq 1$ it follows for $m$ large enough
\begin{align}\label{ineq:amdelta}
a_m\Delta_m^2\geq \frac 1 2 c_m\sqrt{m}|\Delta_m|\rightarrow \infty.
\end{align}
Hence, choosing  as in \eqref{defbm}
\begin{align}\label{defbm2}
	b_m=\sqrt{a_m}/|\Delta_m|
\end{align}
guarantees
\begin{align}\label{eq_bmam}
\frac{b_m}{a_m}=\frac{1}{\sqrt{a_m}|\Delta_m|}\rightarrow 0
\end{align}
such that
\begin{align*}
	\frac{|\Delta_m|\, b_m}{\sqrt{a_m+x b_m}}=\frac{\sqrt{a_m}}{\sqrt{a_m\left( 1+x\,\frac{b_m}{a_m} \right)}}\rightarrow 1.
\end{align*}
This shows that \eqref{eq:cev} is fulfilled for these choices of normalizing sequences. 
By these considerations and Proposition~\ref{delta.lim2} below the proof of Theorem~\ref{thm:cond:version} is complete.

\subsection{Simulations}\label{sec:sim}

\begin{figure}
	\centering
	\subfloat[$\beta=0.25$]{
\includegraphics[width=0.5\textwidth]{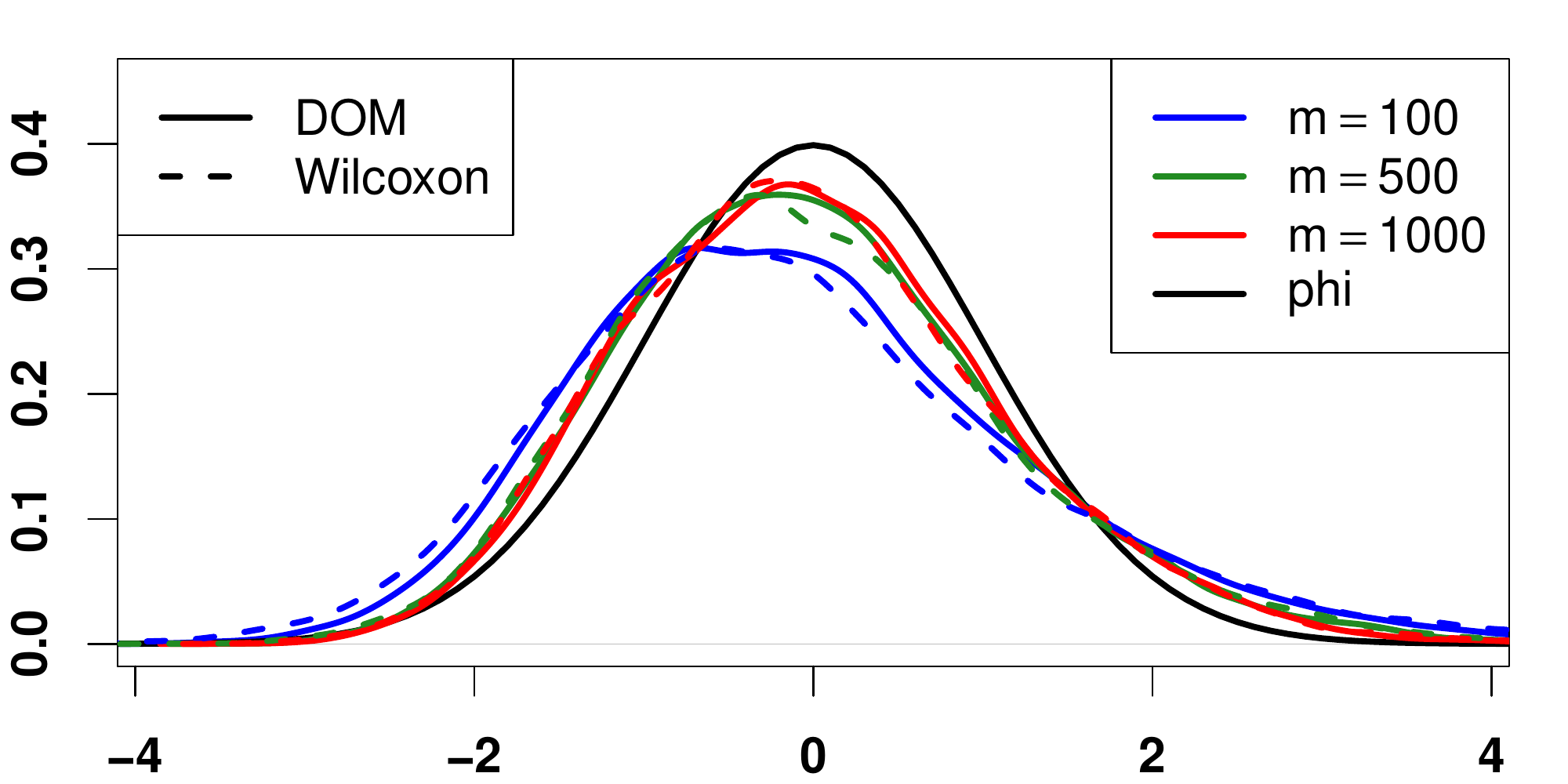}
	}
	\subfloat[$\beta=0.75$]{
\includegraphics[width=0.5\textwidth]{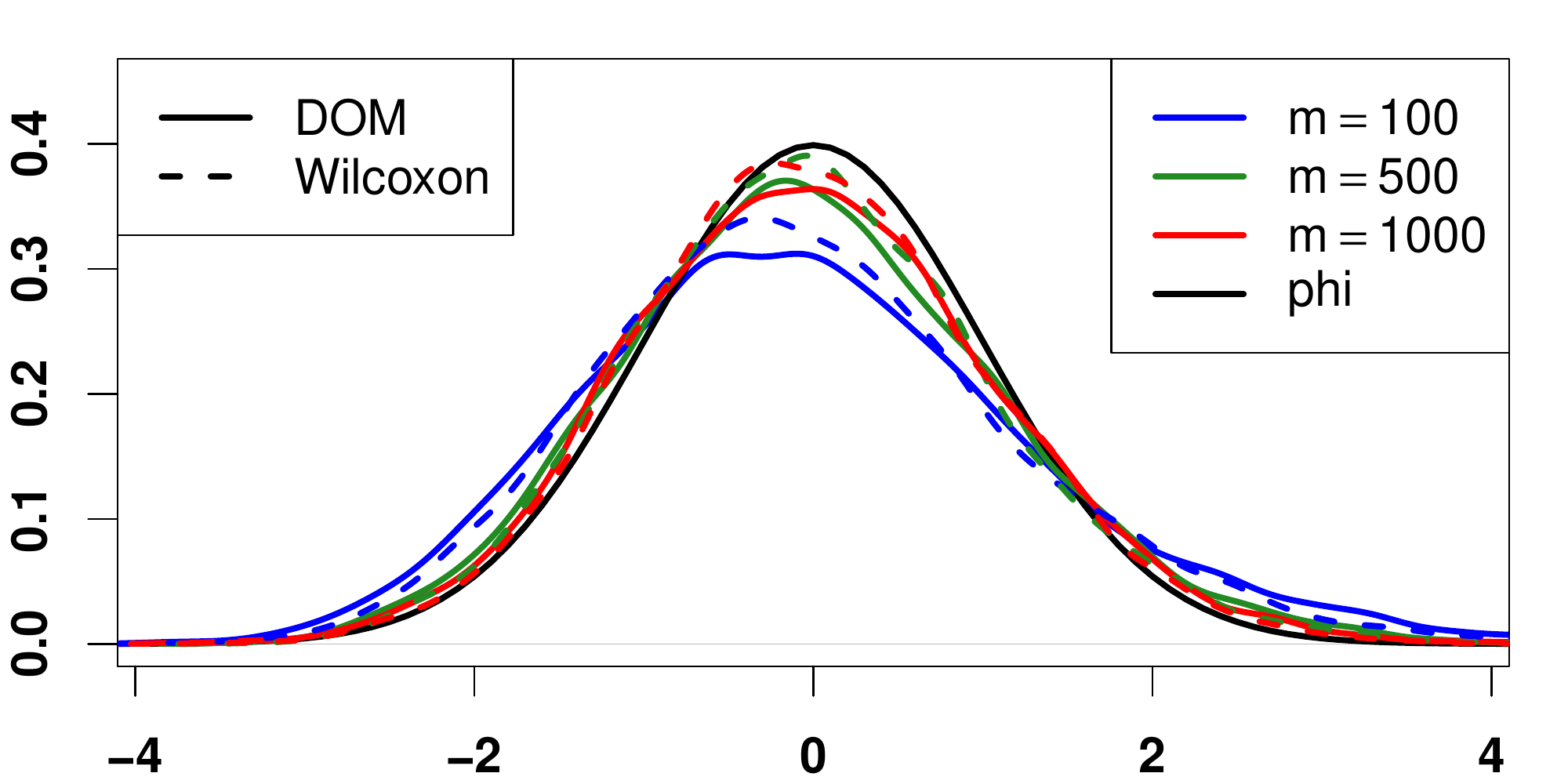}
	}\\
	\subfloat[$\beta=1$]{
\includegraphics[width=0.5\textwidth]{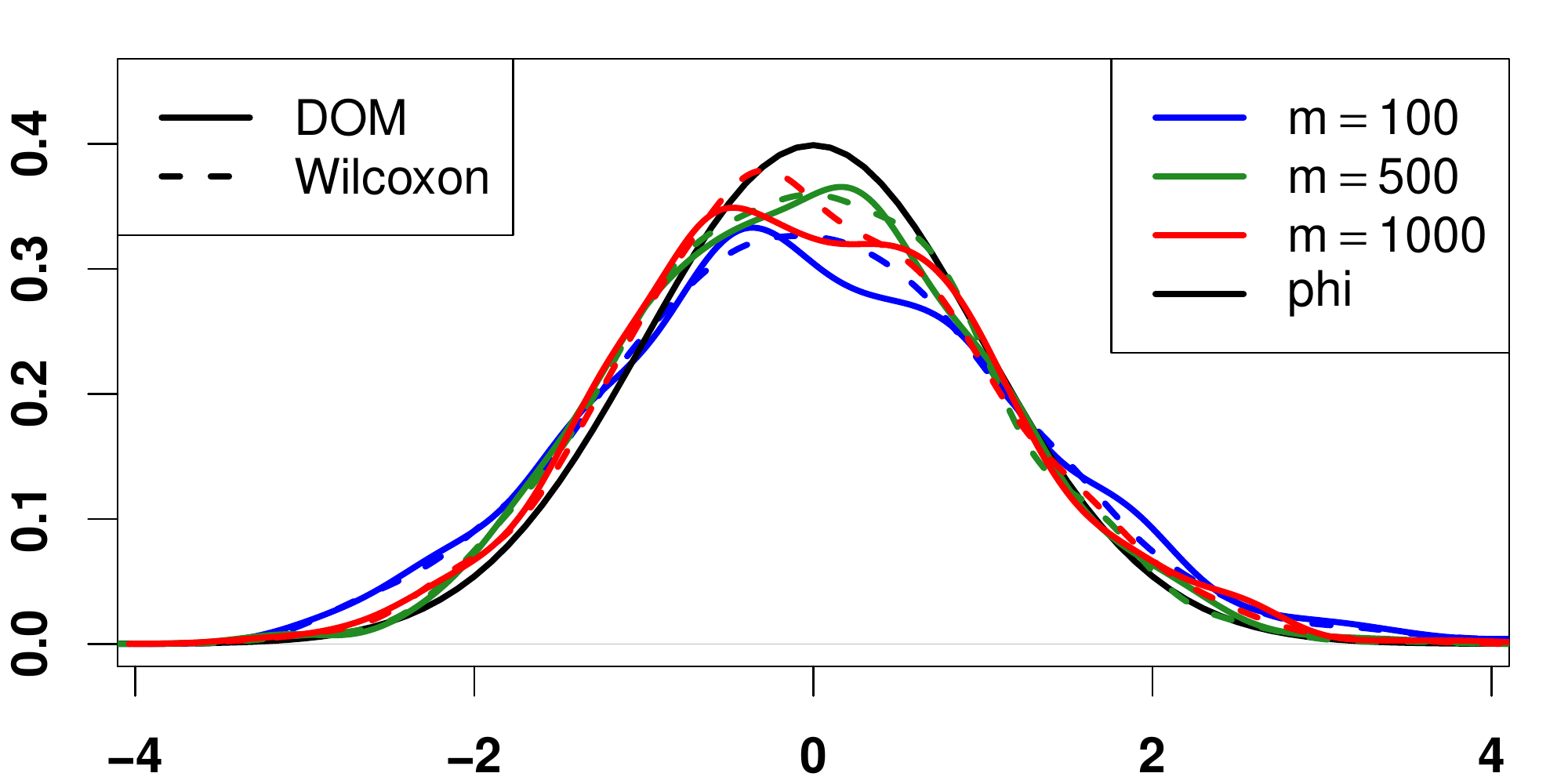}
	}
	\subfloat[$\beta=1.4$]{
\includegraphics[width=0.5\textwidth]{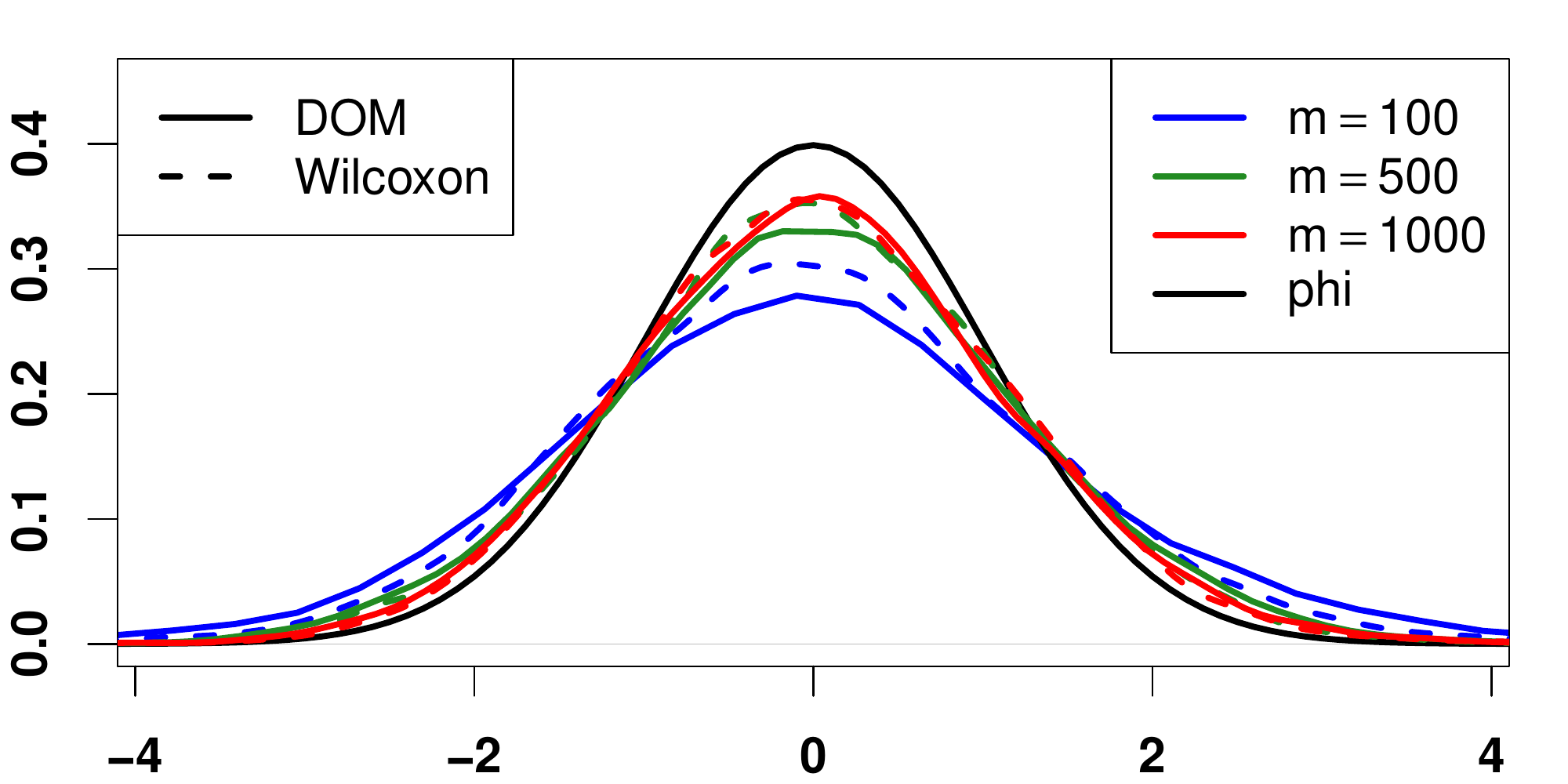}
	}
	\caption{Estimated densities of the standardized stopping times for i.i.d.\ normal data according to Corollary \ref{cor_uncond} for early changes (top row) as well as late changes (bottom row). The dashed lines show the Wilcoxon monitoring scheme while the solid line shows the DOM kernel. Different colors give different monitoring lengths, while the black lines indicates a standard normal density.}
	\label{figure1_neu}
\end{figure}

In Sections \ref{sec_asym} and \ref{sec:imp}, we have shown that the stopping time converges to a Gaussian distribution if standardized appropriately. 
We will now evaluate the fit obtained from that asymptotic approximations in small samples numerically.

To this end we simulate mean changes as in \eqref{meanex} with $d_m=1$.  We use historic data sets of length $m=100,500,1000$ and a monitoring horizon of $30m$. For early changes we use a constant threshold $c^{W/D}_{\alpha}=\sigma_{W/D}q_{1-\alpha}$ fulfilling both Assumptions~\ref{ass_neu_stoch} and \ref{ass_as_S}, where $q_{1-\alpha}$ is the $(1-\alpha)$-quantile of $\sup_{0<t<1}|W(t)|$ with $\alpha=5\%$ (corresponding to a $5\%$-level test as in Remark~\ref{rem:test}), where $q_{1-\alpha}$ is based on $50\,000$ simulations of a Wiener process on a grid of $10\,000$ points.
For late changes we choose $c_m=\sigma_{W/D}\sqrt{2\log\log m}$ such that equality in \eqref{eq_neu_rem_2} holds as $\rho=\sigma_{W/D}$. This choice fulfills Assumption~\ref{ass_neu_stoch} and almost fulfills Assumption \ref{ass_as_S}, see Remark \ref{rem_cm}.

 For better comparability for different kernels as well as different distributions of the underlying time series, we additionally divide the standardized stopping time by $\sigma^{W/D}_{\infty}$ such that the limit distribution is standard normal in all situations.

 We simulate i.i.d.\ Gaussian and $t(3)$ distributed time series as well as AR(1) time series with parameter $a=0.2$ Gaussian and $t(3)$ innovations. In the latter case, $\sigma_{W/D}$ and $\sigma^*_{W/D}$  are replaced by estimators of the long-run covariance which is estimated with a Bartlett kernel based on the historic data set.
$\Delta^D$ and $\Delta^W$ are given by (\ref{delta.cusum}) and (\ref{delta.wil}), where the latter is determined numerically. For time series with  $t$-distributed errors the latter is estimated by means of Monte-Carlo simulations (based on $20\,000$ independent time series of length $10\,000$ each), so is $F$ for the Wilcoxon statistic. All curves below are based on $10\,000$ repetitions.
 
%

Figure~\ref{figure1_neu} shows  estimated densities of the standardized stopping times  for different lengths of the historic data sets for independent standard normally distributed observations along with the standard normal density. The standardization is done as in  Corollary \ref{cor_uncond} including a division by $\sigma_{\infty}$. In all cases the fit is reasonably good and becomes better with increasing length of the historic data set.

\begin{figure}
	\subfloat[i.i.d. $t(3)$ errors]{
		\includegraphics[width=0.33\textwidth]{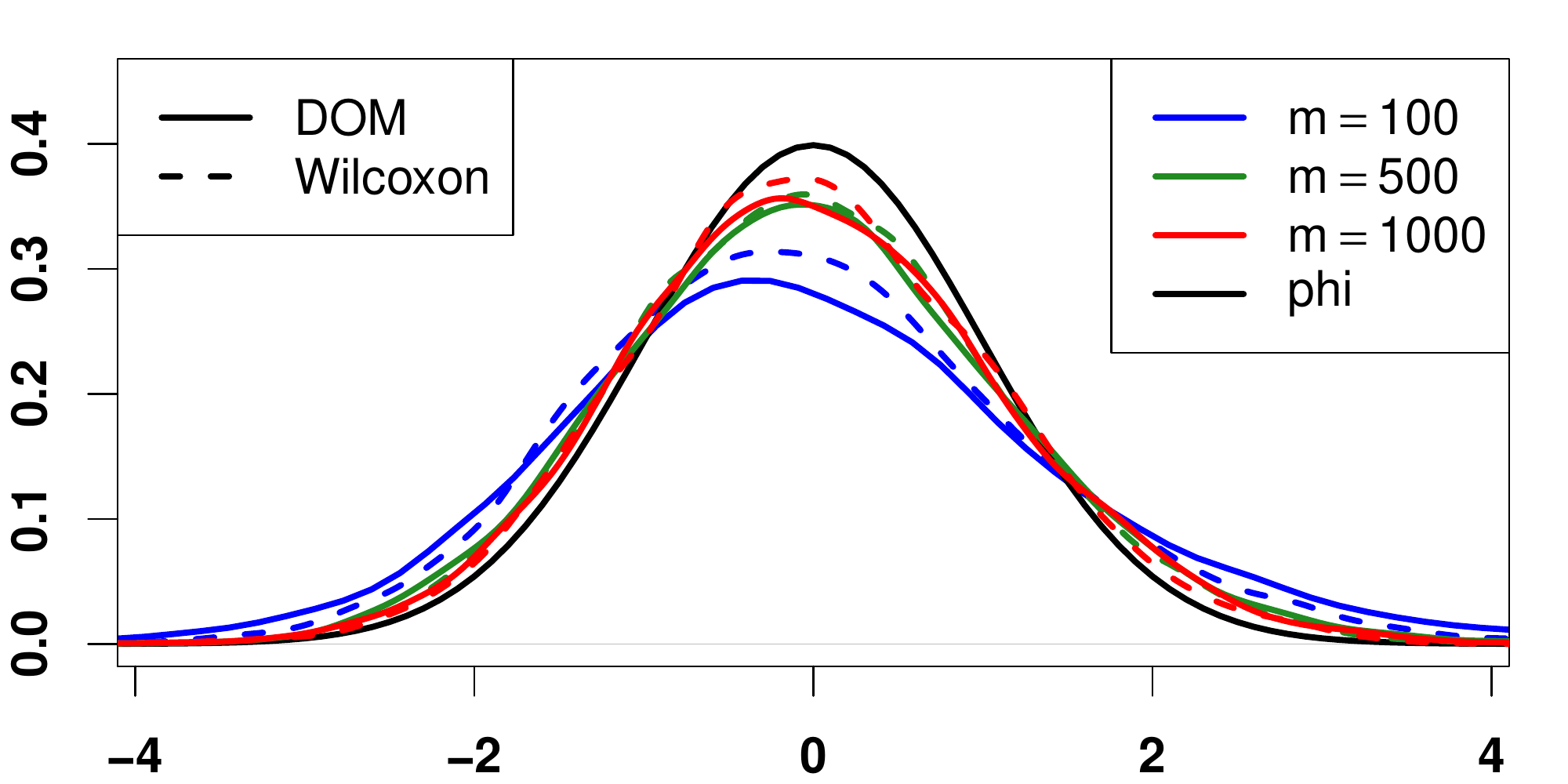}
	}
		\subfloat[AR(1) with $N(0,1)$]{
		\includegraphics[width=0.33\textwidth]{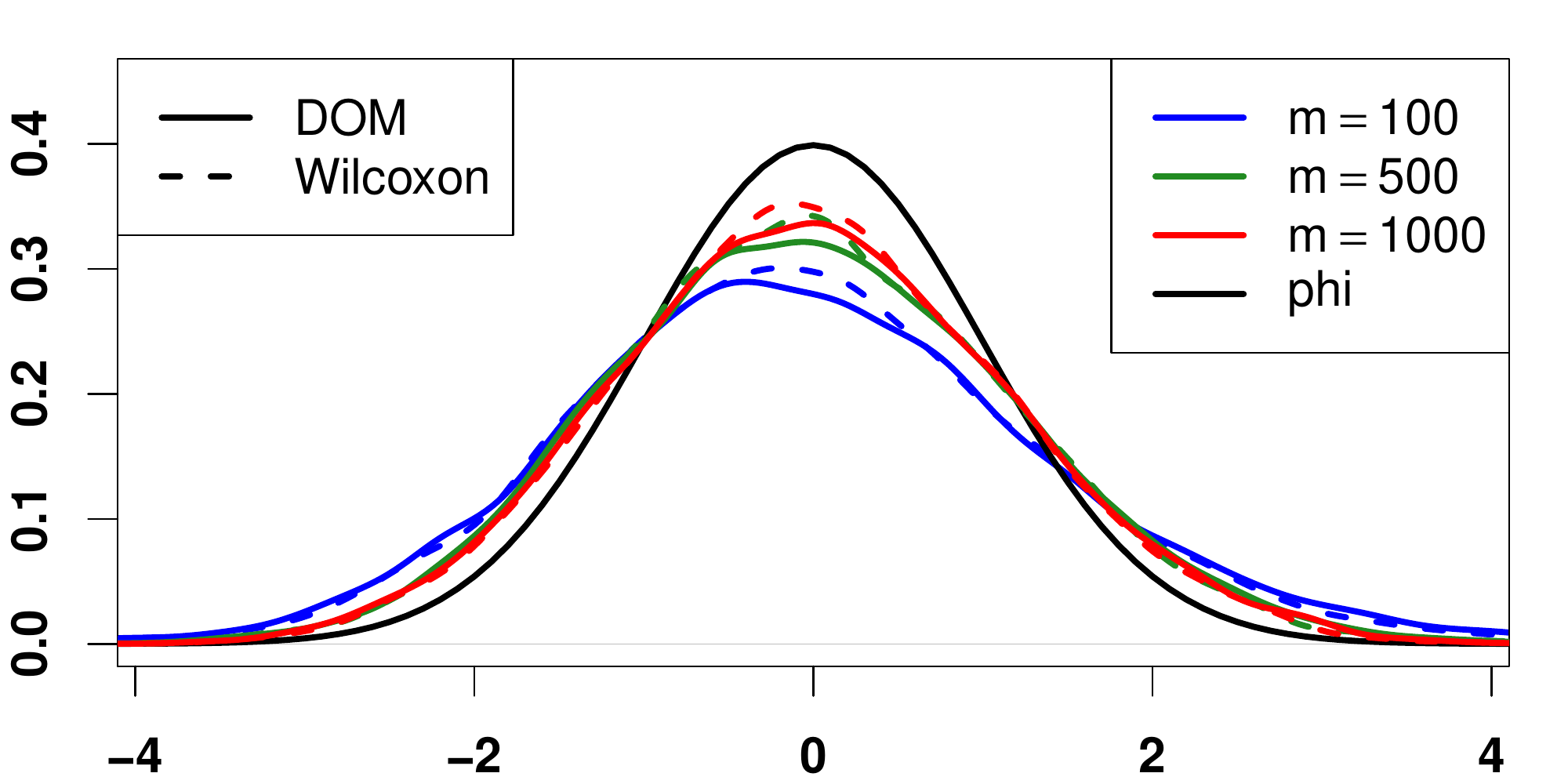}
	}
	\subfloat[AR(1) with $t(3)$]{
		\includegraphics[width=0.33\textwidth]{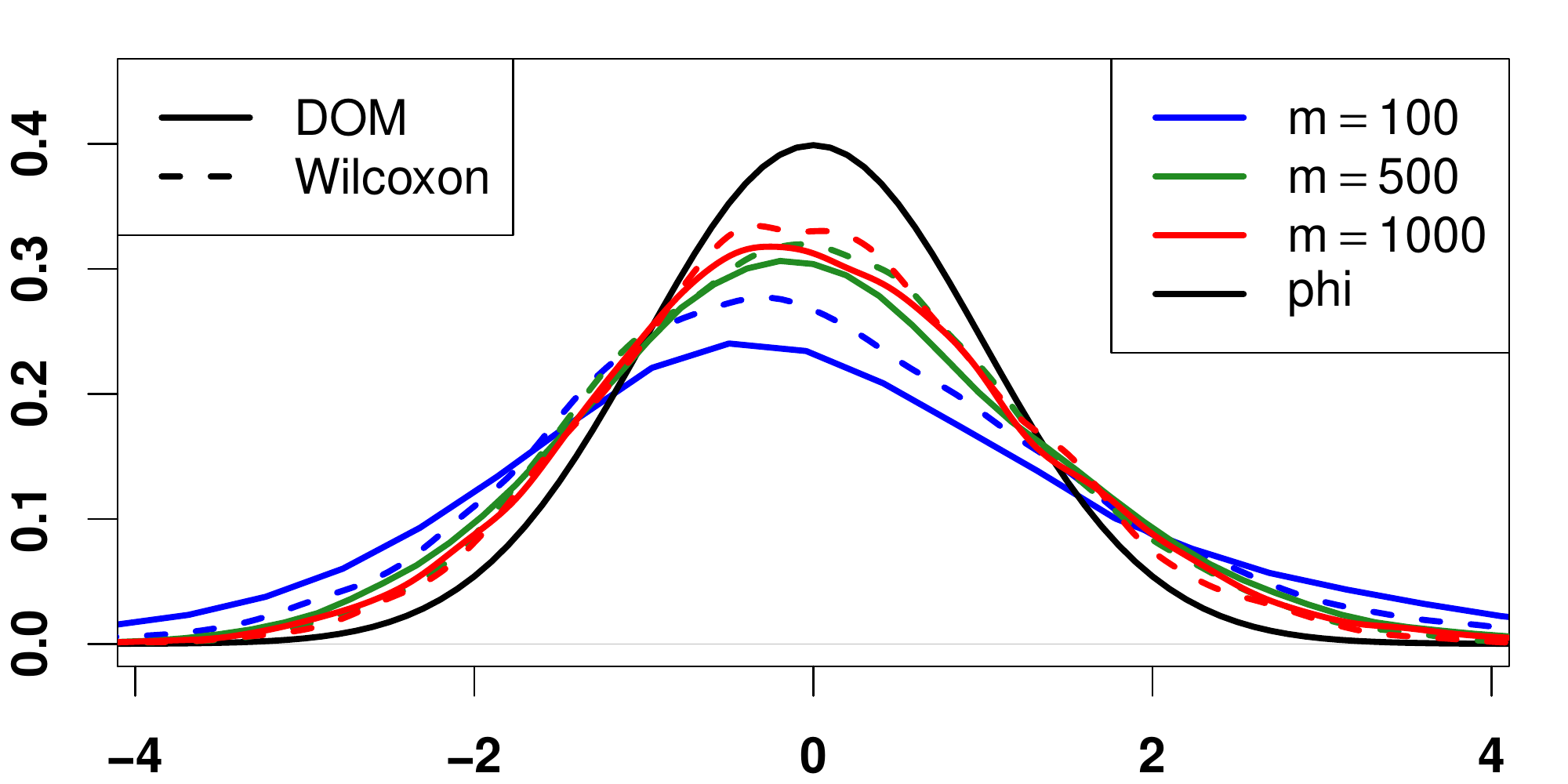}
	}
		\caption{Estimated densities of the standardized stopping times according to Corollary \ref{cor_uncond} for linear ($\beta=1$) changes for (a) i.i.d.\ $t(3)$ distributed errors as well as AR(1) time series with parameter $a=0.2$ errors with (b) normal innovations and (c) $t(3)$ innovations. 		}
		\label{figure2_neu}
\end{figure}

The simulation results for $t(3)$ distributed errors look very similar. For time series errors this is also true, but with a somewhat slower convergence i.e.\ somewhat longer historic data sets are required to get similar results. This is not surprising given that the effective sample  size is smaller for positively correlated time series errors and at the same time the long-variance has been estimated for each time series while the true variance has been used for i.i.d.\ data. To illustrate these points, Figure~\ref{figure2_neu} shows the corresponding results for $\beta=1$ (where results for different values of $\beta$ look very similar).

Simulation results for all cases if $b_m(S_{1,m},S^*_{1,m})$ is replaced by $b_m(0,0)$ as in Theorem~\ref{thm:aprime} a) as well as for sublinear and linear changes if $a_m(S_{1,m},S^*_{1,m})$ is replaced by $a_m(S_{1,m},0)$ as in Theorem~\ref{thm:aprime} b) look very similar. For sublinear changes by Theorem~\ref{thm:aprime} c) the asymptotic distribution is independent of both $S_{1,m}$ and $S_{1,m}^*$ such that $a_m(0,0)$ and $b_m(0,0)$ can be used. To illustrate this point Figure~\ref{figure3_neu} (a) and (b) give the corresponding simulation results for $\beta=0.5$ for i.i.d.\ normal errors. Clearly, the fit with the additional knowledge of $S_{1,m}$ and $S_{1,m}^*$ is somewhat better than the one without that additional information, but the difference is not very large.
As indicated by Remark~\ref{rem_1_imp} the factor in $a_m(0,0)$ is asymptotically negligible in this case and  has not been included in the previous literature. However, as seen in Figure~\ref{figure3_neu} (c) even for a historic sample length of $m=1000$ a large bias remains without this factor.

\begin{figure}
	\subfloat[$a_m(S_{1,m},S^*_{1,m})$, $b_m(S_{1,m},S^*_{1,m})$]{
		\includegraphics[width=0.33\textwidth]{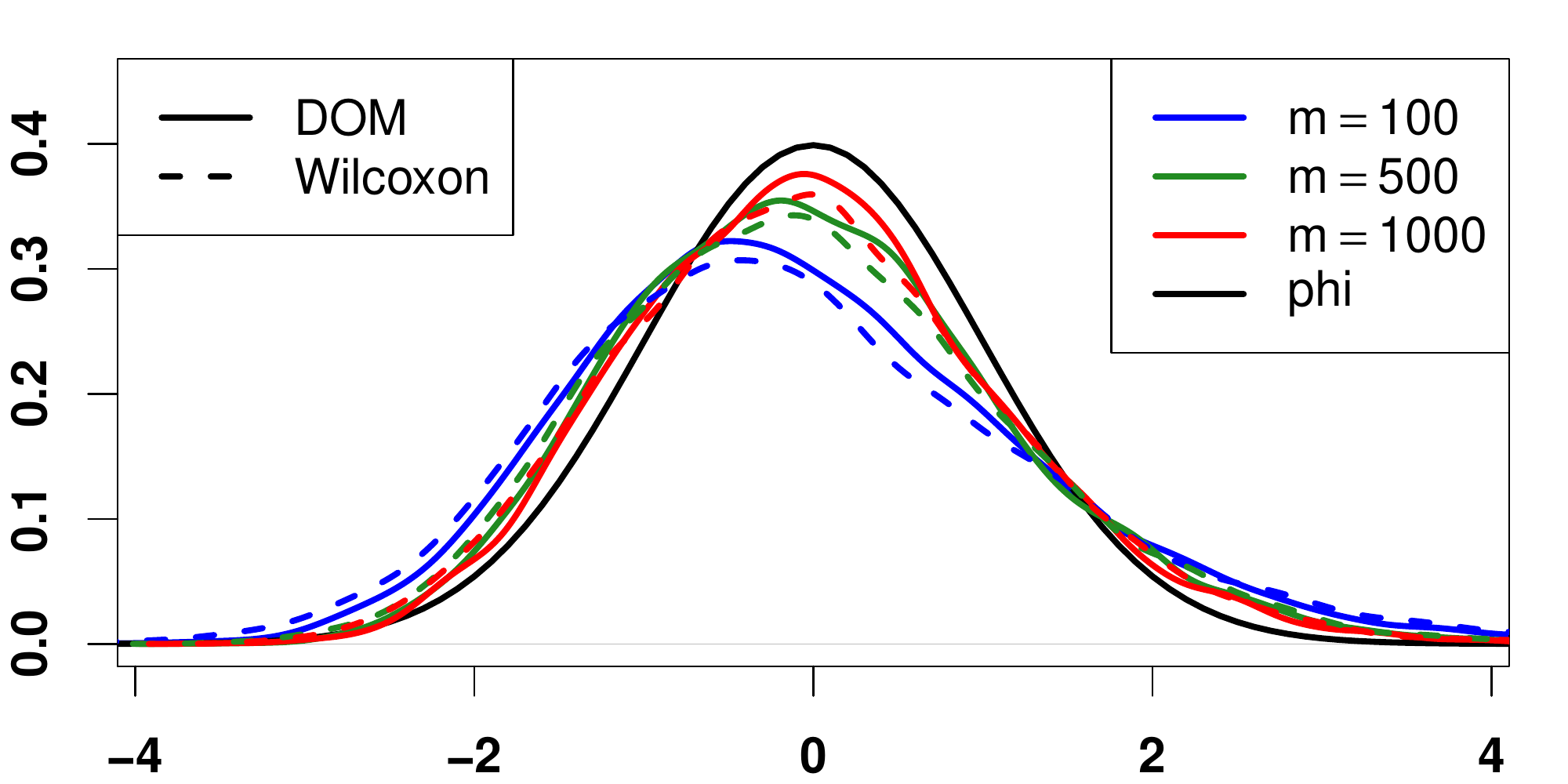}
	}
		\subfloat[$a_m(0,0)$, $b_m(0,0)$]{
		\includegraphics[width=0.33\textwidth]{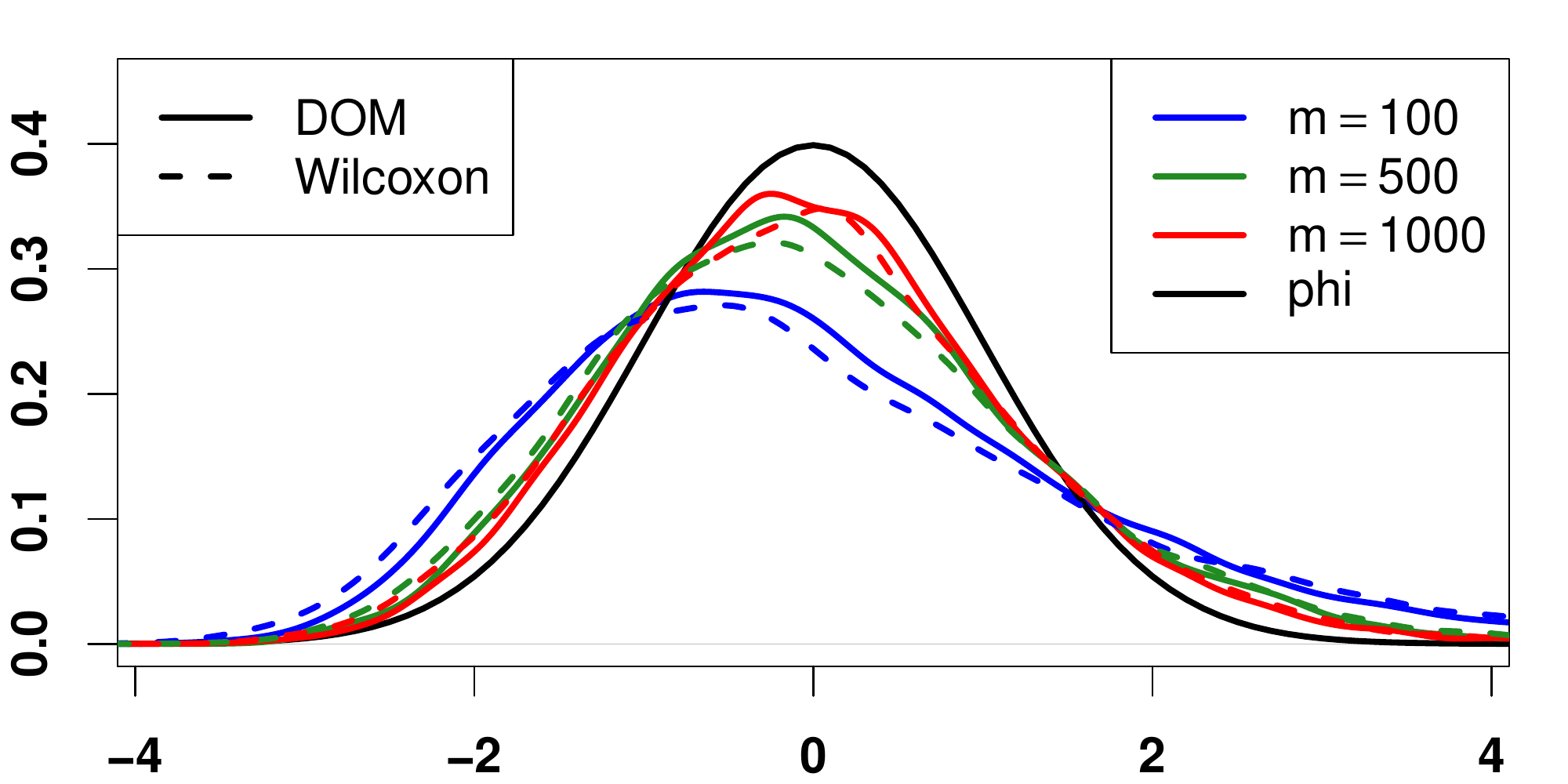}
	}
	\subfloat[no bias correction]{
		\includegraphics[width=0.33\textwidth]{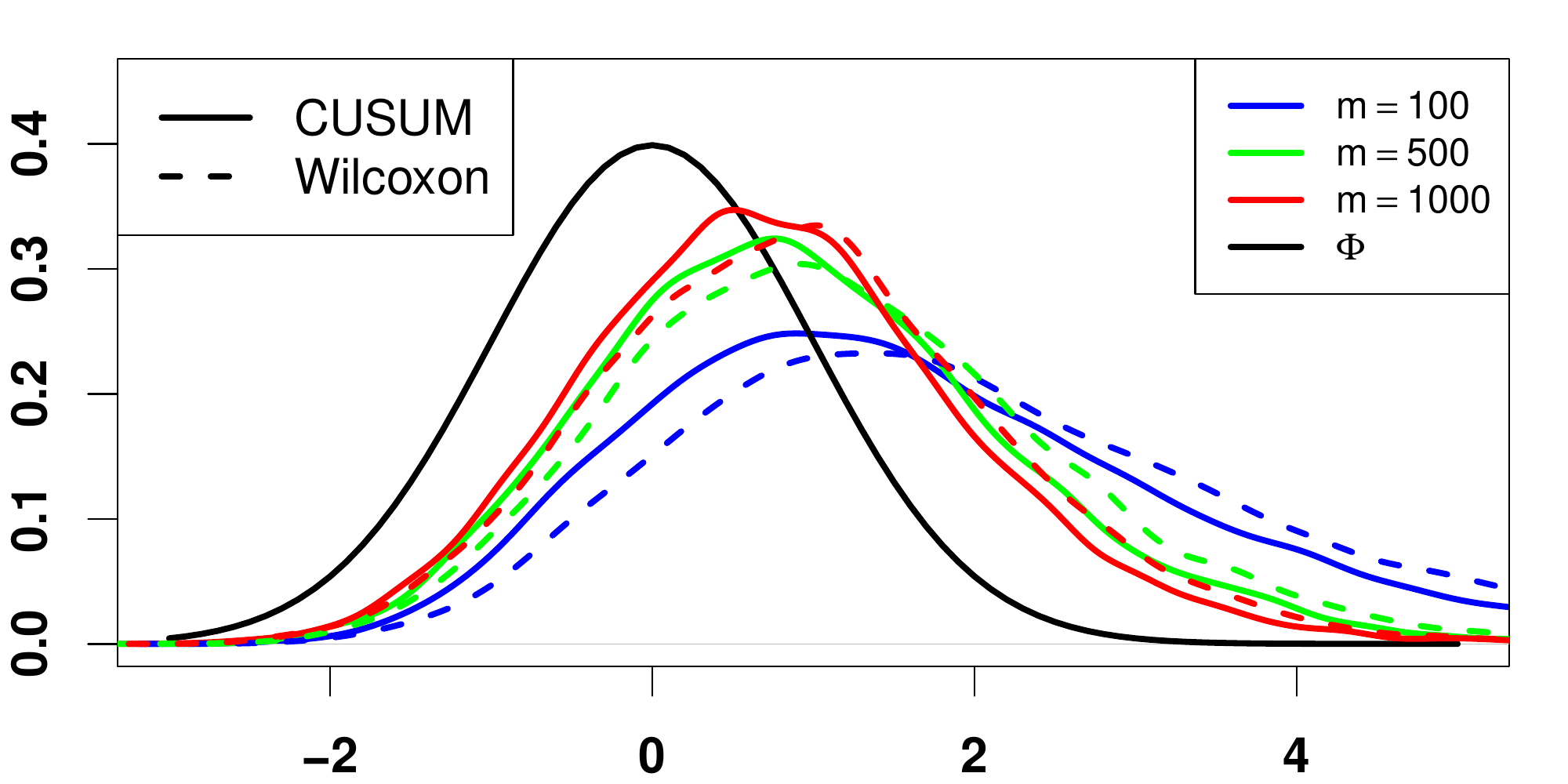}	}
	\caption{Estimated densities of the standardized stopping times for early changes with $\beta=0.5$ with i.i.d.\ normal errors. The standardization was chosen in (a) according to Corollary \ref{cor_uncond} with $a_m(S_{1,m},S^*_{1,m})$ and $b_m(S_{1,m},S^*_{1,m})$, in (b) and (c) according to Theorem~\ref{thm:aprime} c) with $a_m(0,0)$ and $b_m(0,0)$, where in (c) the bias-correcting factor $(1-\frac{c_m}{\sqrt{m}|\Delta_m|})^{-1}$ has been suppressed (see Remark~\ref{rem_1_imp}).
	}
		\label{figure3_neu}
\end{figure}

\section{DOM versus Wilcoxon sequential change point procedures}\label{section_COM_Wil}
In this section, we use the asymptotic results from the previous section to compare the detection delay of the DOM and the Wilcoxon sequential change point procedures. We then confirm these results by means of simulations.

\subsection{Theoretical comparison}

\begin{table}
\renewcommand{\arraystretch}{1.5}
\begin{tabular}{c|c|c|c|}
&N(0,1)&\mbox{Laplace}(0,1)&t(3)\\
\hline
$\frac{\sigma_W}{\left|\Delta_m^W\right|}-\frac{\sigma_D}{\left|\Delta_m^D\right|}$&0.109 &-0.126& -0.379
\end{tabular}
\caption{Difference in signal-to-noise ratio for different distributions.}
\label{scalediff}
\end{table}

In this section we compare the delay $a_m(0,0)$ as in \eqref{eq_am_00} for the two methods. For early changes this corresponds to the expected delay as given in Theorem~\ref{thm:aprime} c), while for late changes we have replaced the random quantities $S_{1,m}$ and $S_{1,m}^*$ in $a_m(S_{1,m},S_{1,m}^*)$ by their expected values.

Consider critical values as in the previous section $c_m=\sigma_{W/D}\,\tilde{c}_m$ with $\tilde{c}_m$ being equal to the asymptotic $(1-\alpha)$-quantile of $\sup_{0\le t\le 1}|W(t)|$ for early changes and to $\sqrt{2\log\log m}$ for late changes.

Some direct calculations show that 
\begin{align}
	a_m^W(0,0)-a_m^D(0,0)&=\left(\frac{\sigma_W}{\left|\Delta_m^W\right|}-\frac{\sigma_D}{\left|\Delta_m^D\right|}\right) \tilde{c}_m \sqrt{m}\left( 1+\frac{k^*}{m} \right)
\left(1-\frac{\sigma_{W}\,\tilde{c}_m}{\sqrt{m}|\Delta_m^{W}|}\right)^{-1}\,\left(1-\frac{\sigma_{D}\,\tilde{c}_m}{\sqrt{m}|\Delta_m^{D}|}\right)^{-1}
	\notag\\
	&= \left(\frac{\sigma_W}{\left|\Delta_m^W\right|}-\frac{\sigma_D}{\left|\Delta_m^D\right|}\right) \tilde{c}_m \sqrt{m}\left( 1+\frac{k^*}{m} \right) (1+o(1)).
	\label{eq_diff_stop}
\end{align}
Clearly, the difference in signal-to-noise ratio $\frac{\sigma_W}{\left|\Delta_m^W\right|}-\frac{\sigma_D}{\left|\Delta_m^D\right|}$ as given in Table~\ref{scalediff} determines which procedure detects changes quicker and quantifies by how much (relative to how late the change is and how long the historic data set is). This shows that -- as expected -- the Wilcoxon detects changes faster for more heavy-tailed distributions such as e.g.\ the $t(3)$ and Laplace distribution, while the DOM is preferable for Gaussian data.  Futhermore, the factor $1+\frac{k^*}{m}$ indicates that the difference is bigger for later changes which is not asymptotically negligible for late changes. Finally, a more detailed calculation shows that the term $1+o(1)$ is in fact always strictly larger than 1 such that a bias in the same direction as indicated by the difference in signal-to-noise ratio occurs and the true difference can be expected to be somewhat larger than indicated by the above expression.

\begin{figure}[b]
	\subfloat[$N(0,1)$, $\beta=0.25$]{
		\includegraphics[width=0.32\textwidth]{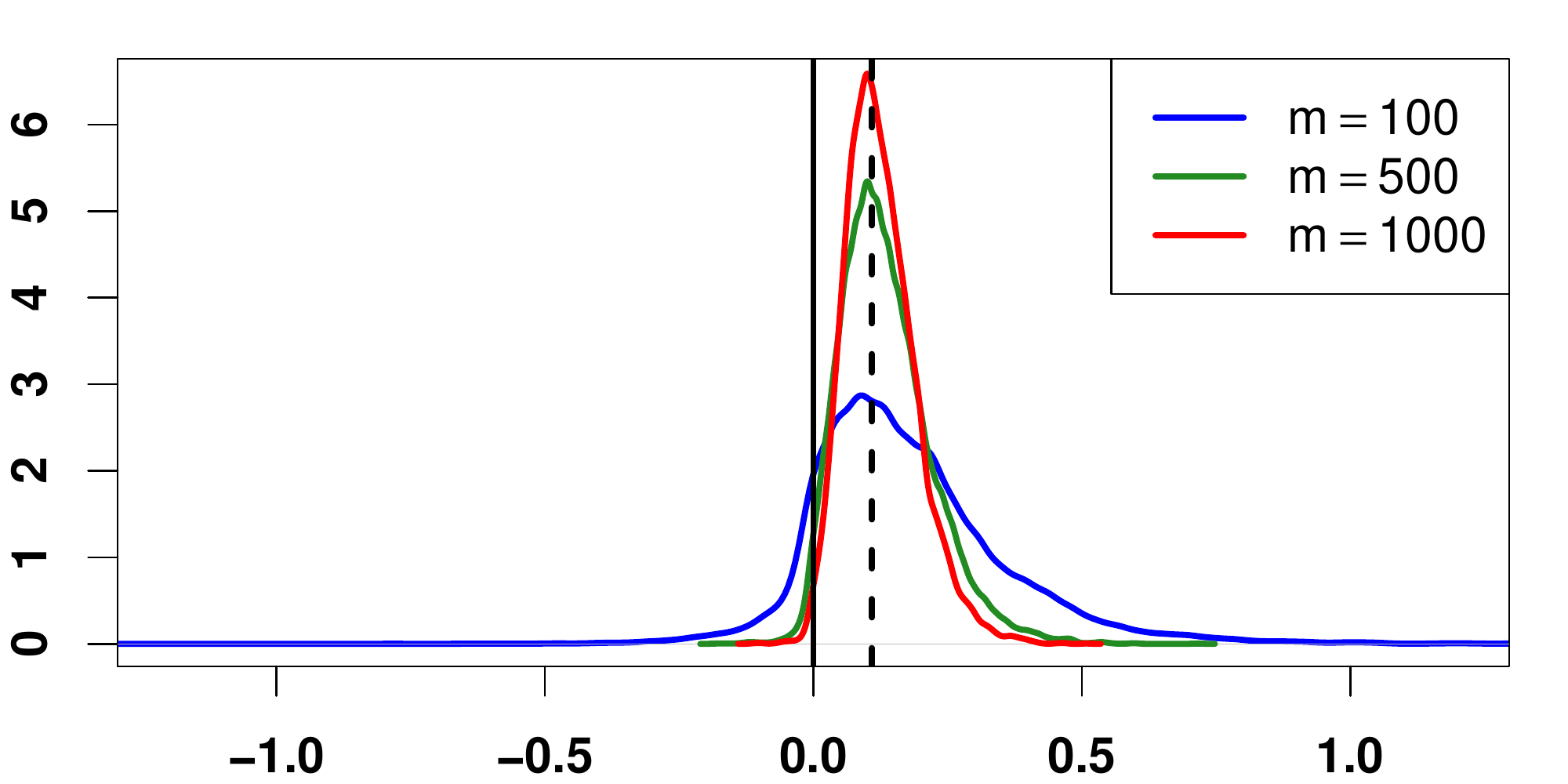}
		}
	\subfloat[$N(0,1)$, $\beta=0.75$]{
		\includegraphics[width=0.32\textwidth]{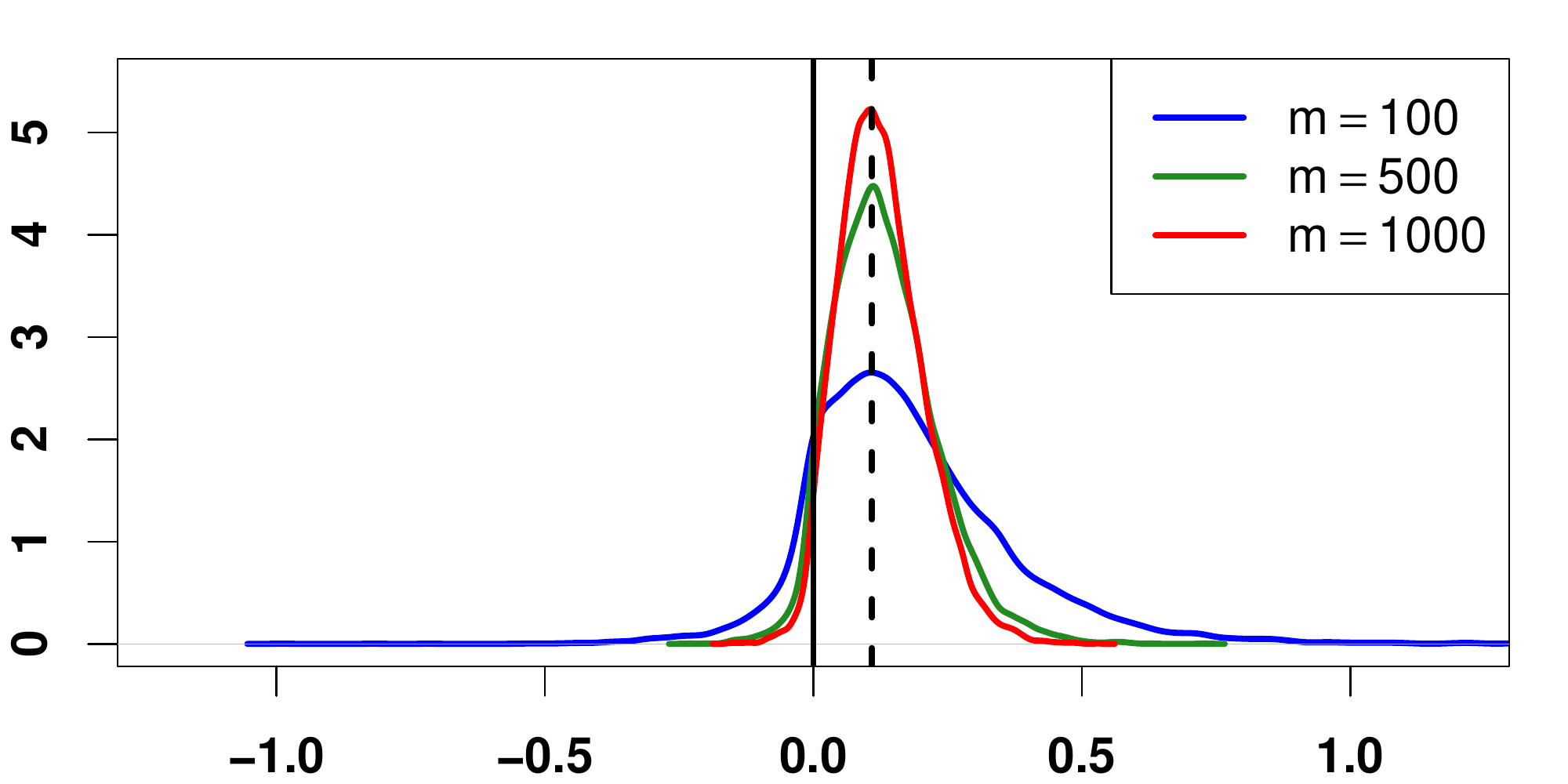}
	}
	\subfloat[$N(0,1)$, $\beta=1.4$]{
		\includegraphics[width=0.32\textwidth]{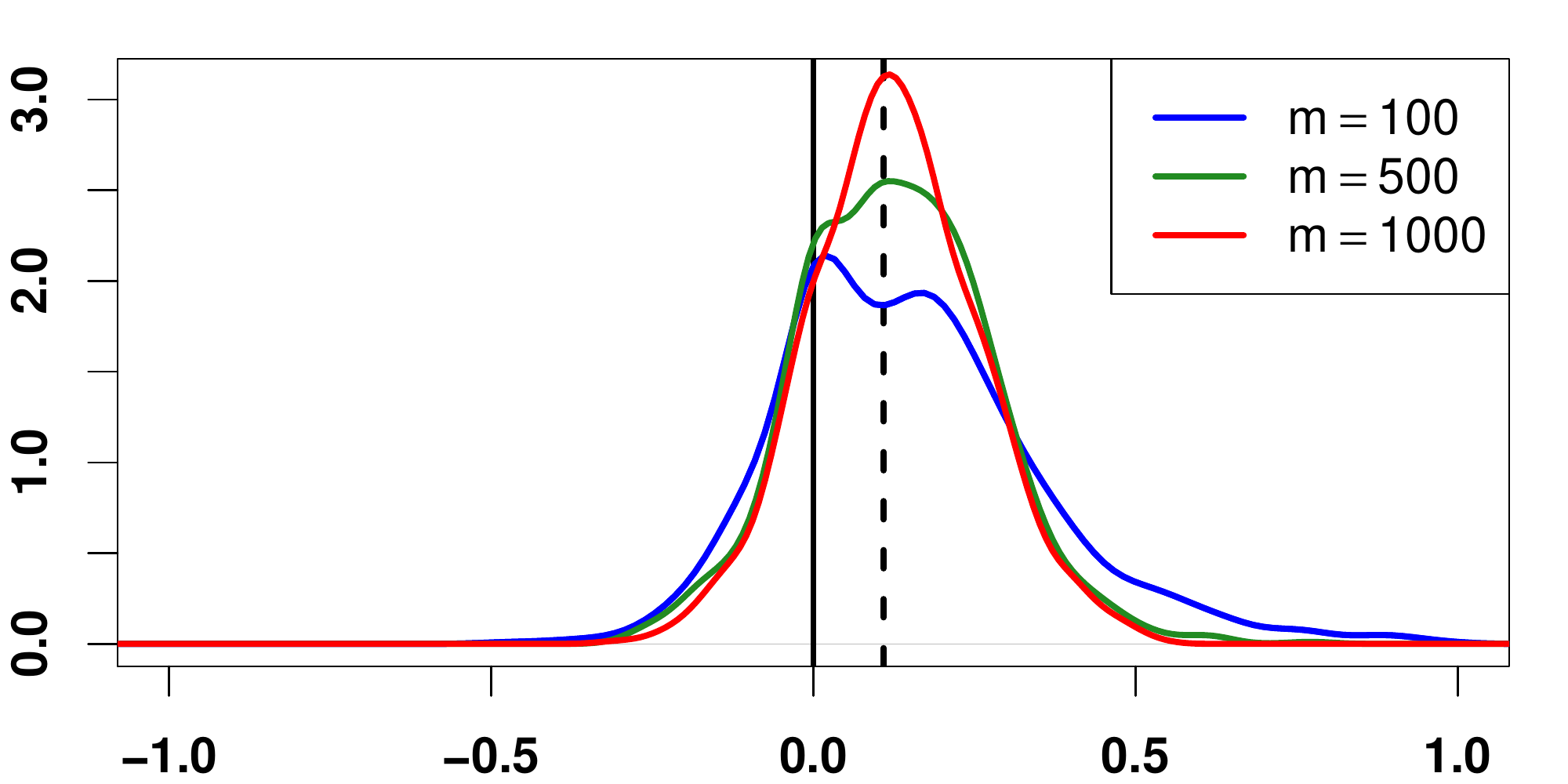}
	}\\
	\subfloat[$t(3)$, $\beta=0.25$]{
		\includegraphics[width=0.32\textwidth]{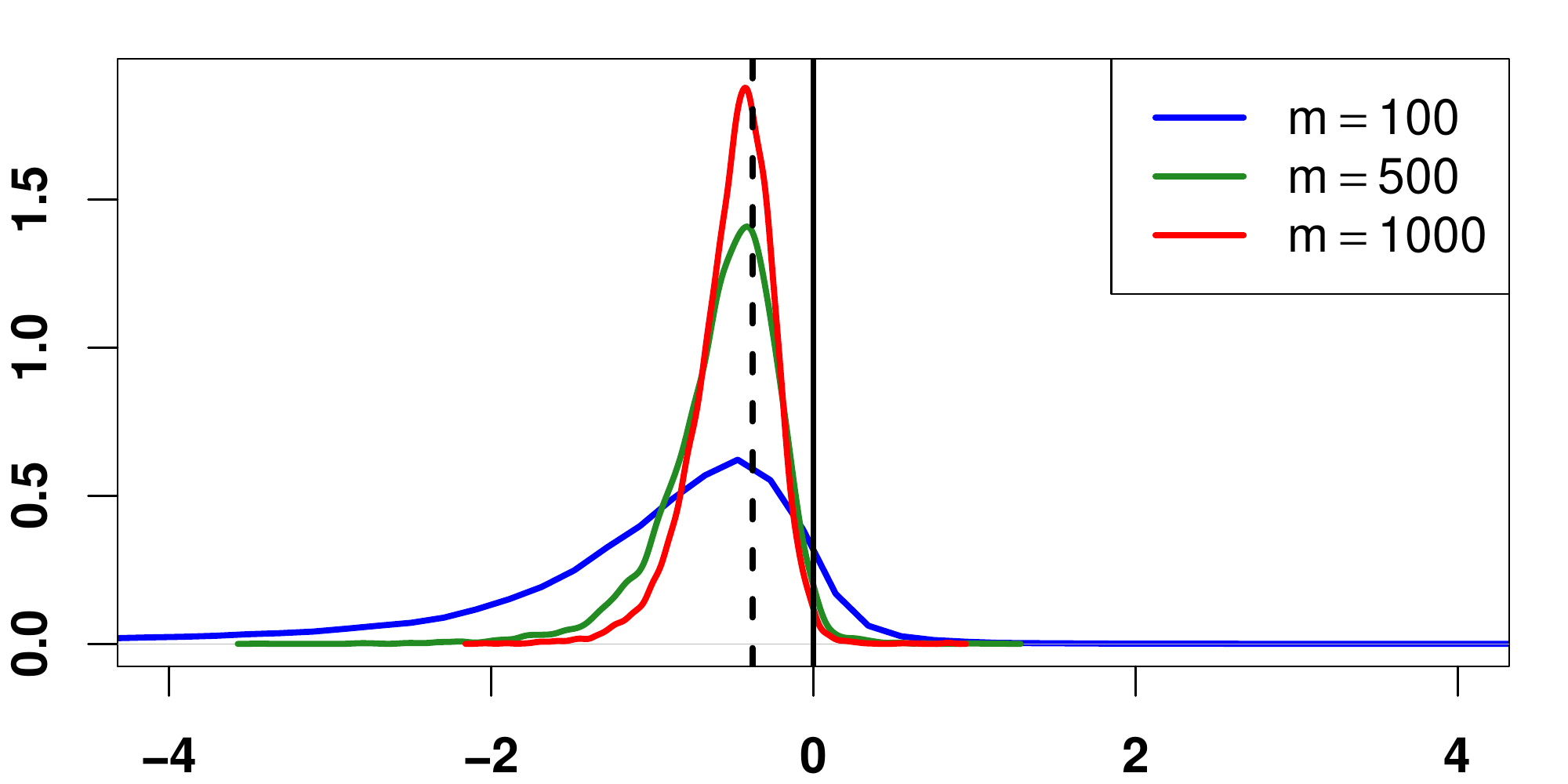}
		}
	\subfloat[$t(3)$, $\beta=0.75$]{
		\includegraphics[width=0.32\textwidth]{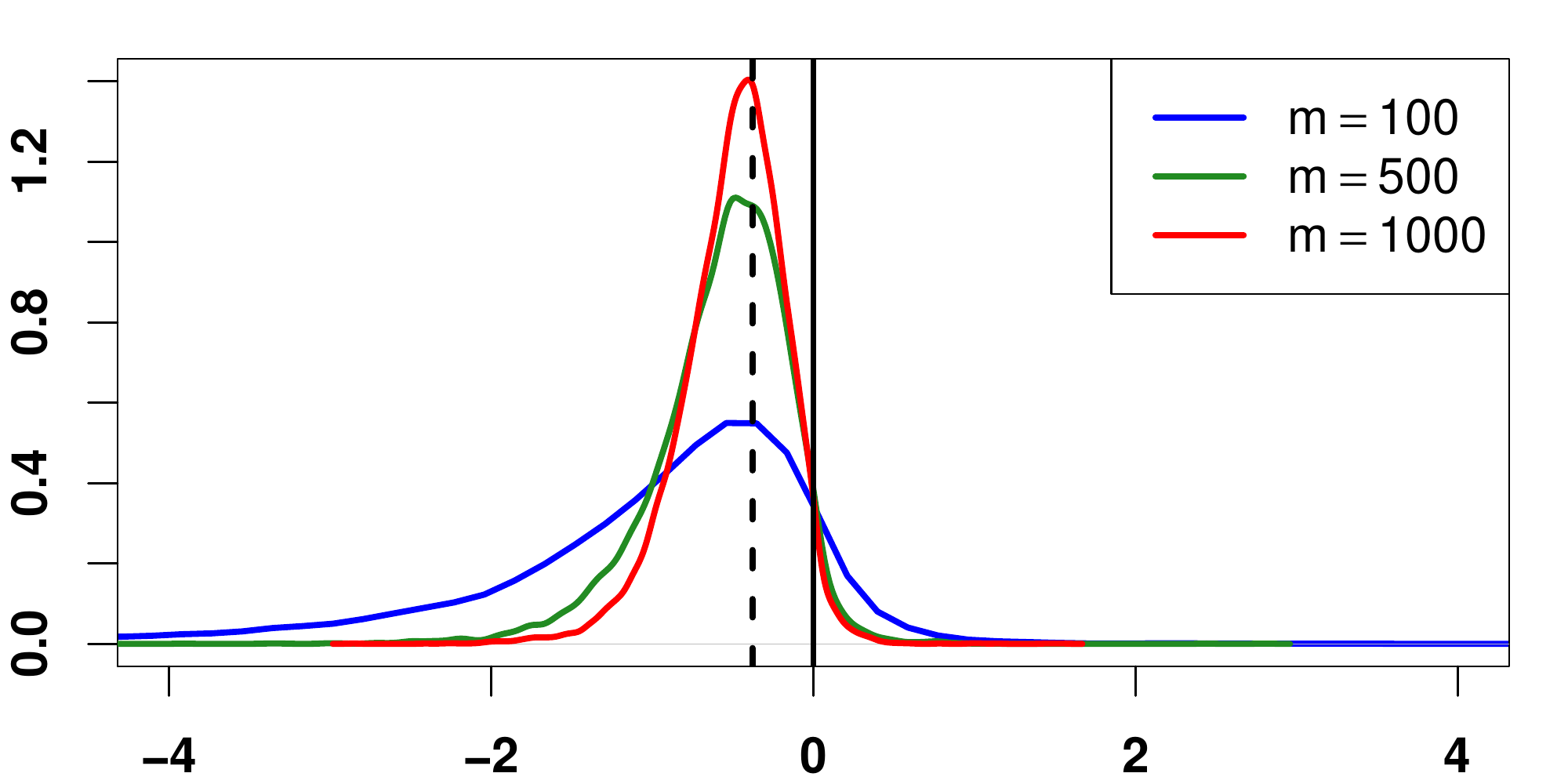}
	}
	\subfloat[$t(3)$, $\beta=1.4$]{
		\includegraphics[width=0.32\textwidth]{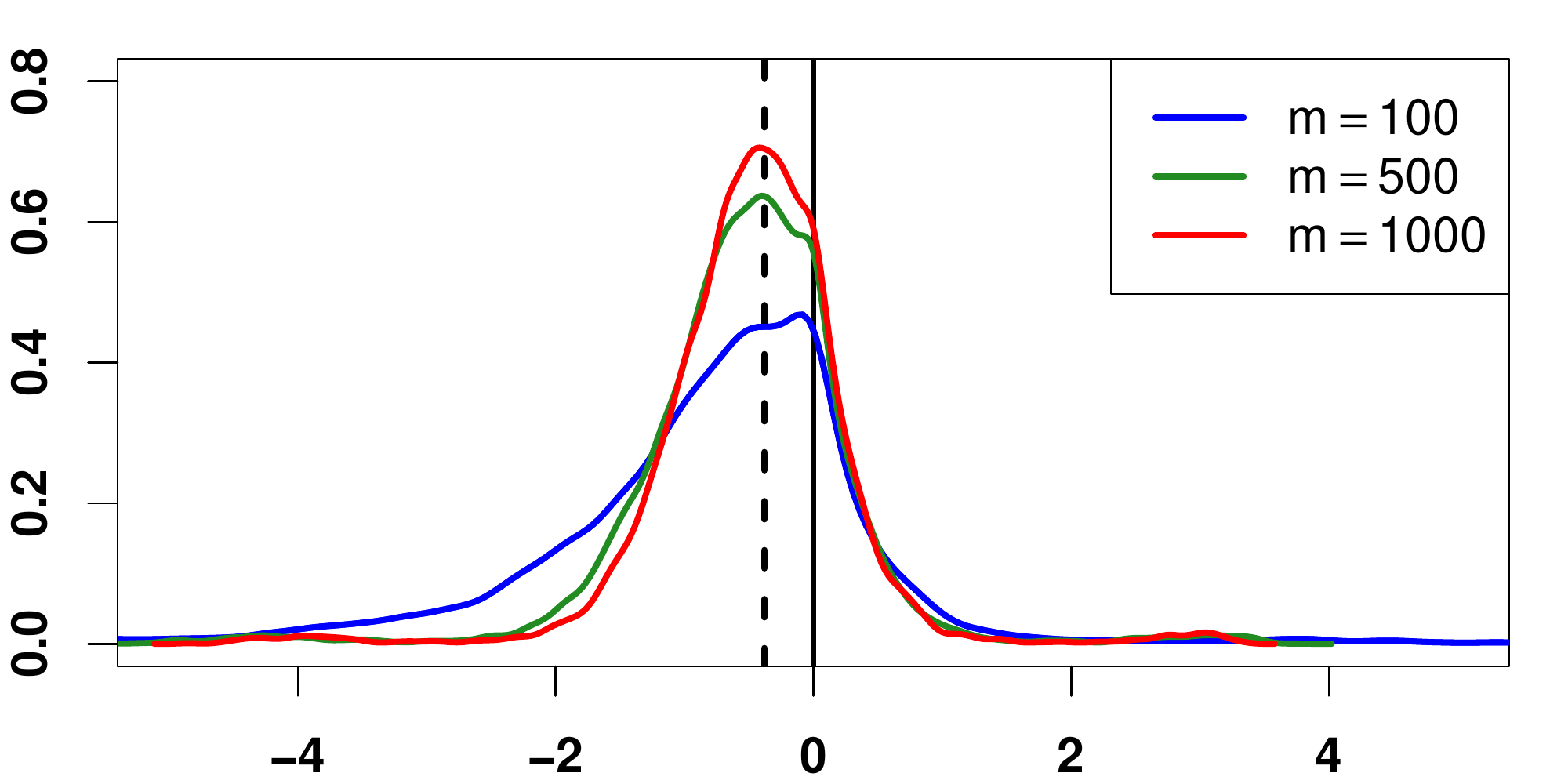}
	}

		\caption{Estimated densities 
			 of the observed delay times scaled by $\tilde{c}_m \sqrt{m}\left( 1+\frac{k^*}{m} \right)$ for i.i.d.\ normal (top row) as well as $t(3)$ (bottom row) data. The dashed line indicates the theoretical value from Table \ref{scalediff}, while the solid line indicates where 0 lies. 		}
		\label{figure_new_WD}
\end{figure}

\subsection{Comparison based on simulations}
In this section, we show simulations indicating the difference in stopping time for the two procedures. In order to have comparable results for different lengths of the historic data set, we standardize both $a_m^W(0,0)-a_m^D(0,0)$ as well as the actually difference of the observed delay time  by $\tilde{c}_m \sqrt{m}\left( 1+\frac{k^*}{m} \right)$ see \eqref{eq_diff_stop}.

Exemplary estimated densities of the observed difference for i.i.d.\ data can be found in Figure~\ref{figure_new_WD}.
Positive values indicate that the DOM procedure was faster while negative values indicate that the Wilcoxon procedure was faster, where for better readability there is a vertical solid line at 0. The dashed line indicates the theoretic value as given in Table~\ref{scalediff}.

 For the standard normal distribution, the estimated densities are well concentrated around the theoretical value showing that the DOM procedure detects changes more quickly than the Wilcoxon procedure. Opposite behavior can be observed for the heavy-tailed $t$ distribution as predicted by the negative value in Table \ref{scalediff}. In all cases the predicted bias in the direction towards the theoretical quantile can be seen. Keeping in mind that the actual delay times in all plots have been divided by $(1+k^*/m)$, it is clear that the advantage of one procedure over the other is strongly increasing the later the change occurs.
 Taking a closer look at the values on the $x$-axis indicates that potential  gain from using the Wilcoxon kernel in case of a $t(3)$ distribution is much larger than the loss in the normal case as predicted by Table~\ref{scalediff}.

 For time series errors as above (i.e.\ an AR(1) time series with parameter $a=0.2$ and Gaussian as well as $t(3)$ innovations) the general tendency remains true, where the signal-to-noise ratio is no longer given by the above table. Figure~\ref{figure_new_WD2} shows the corresponding results.  

\begin{figure}
	\subfloat[$N(0,1)$, $\beta=0.25$]{
		\includegraphics[width=0.32\textwidth]{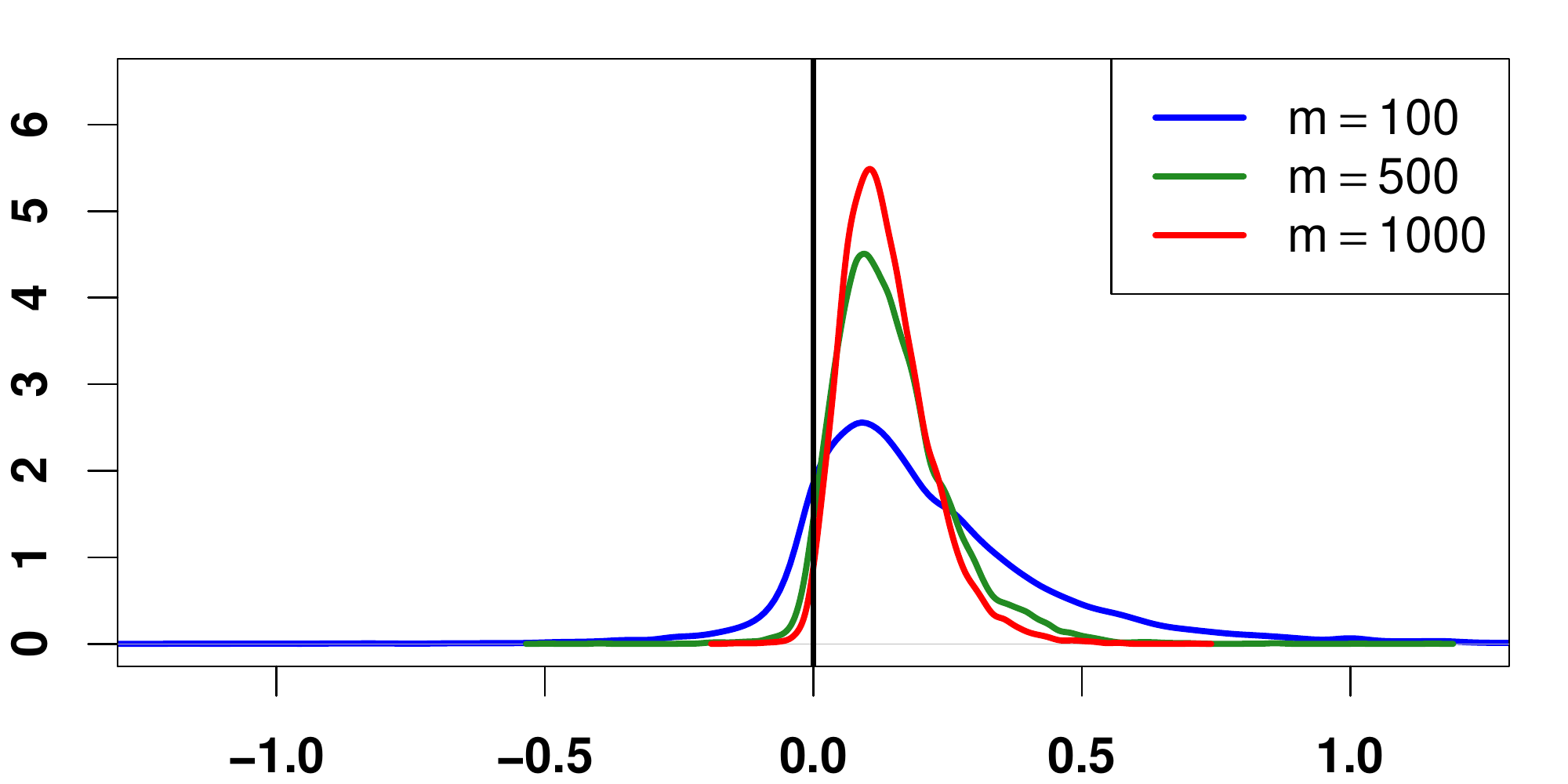}
		}
	\subfloat[$N(0,1)$, $\beta=0.75$]{
		\includegraphics[width=0.32\textwidth]{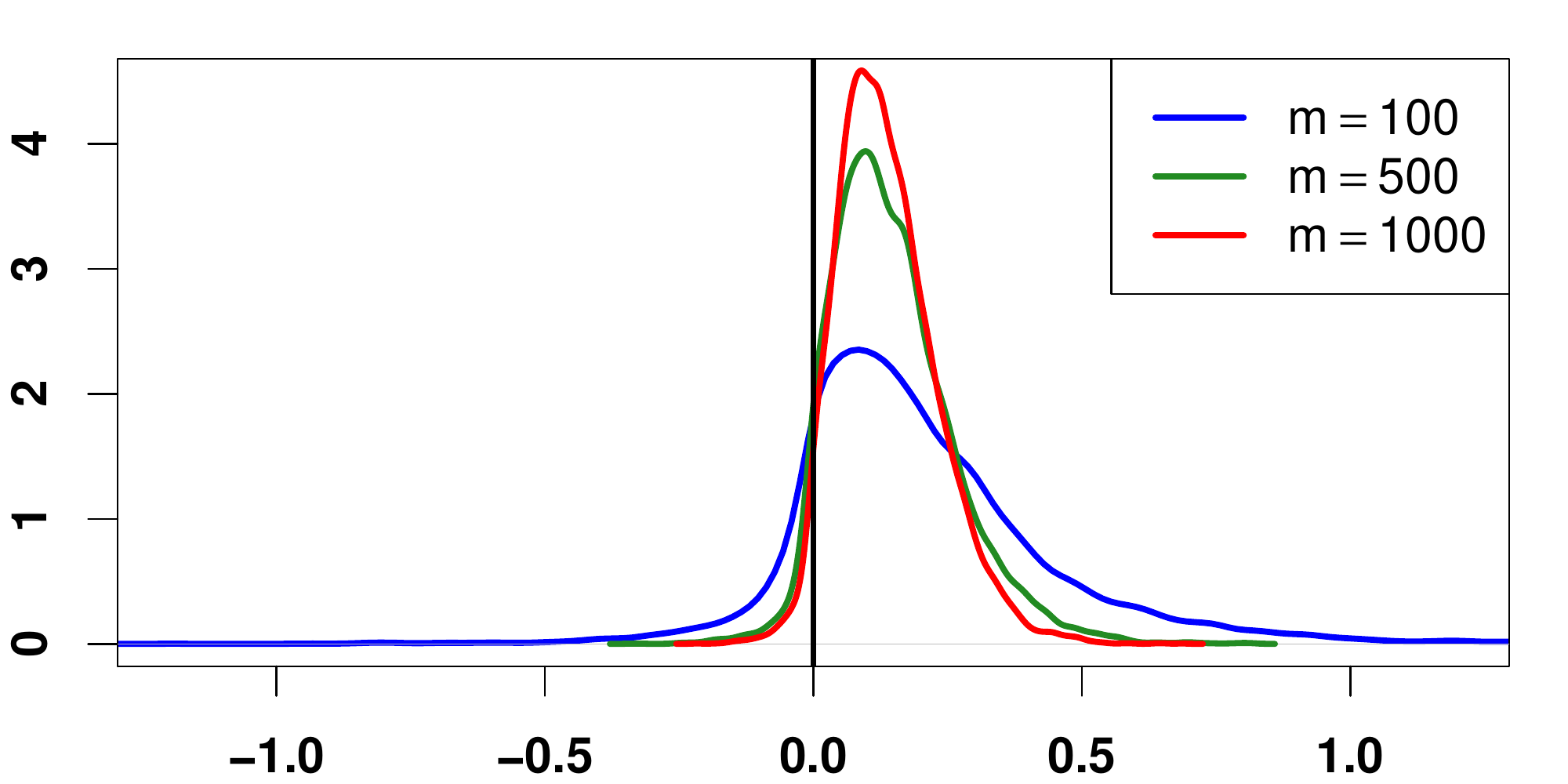}
	}
	\subfloat[$N(0,1)$, $\beta=1.4$]{
		\includegraphics[width=0.32\textwidth]{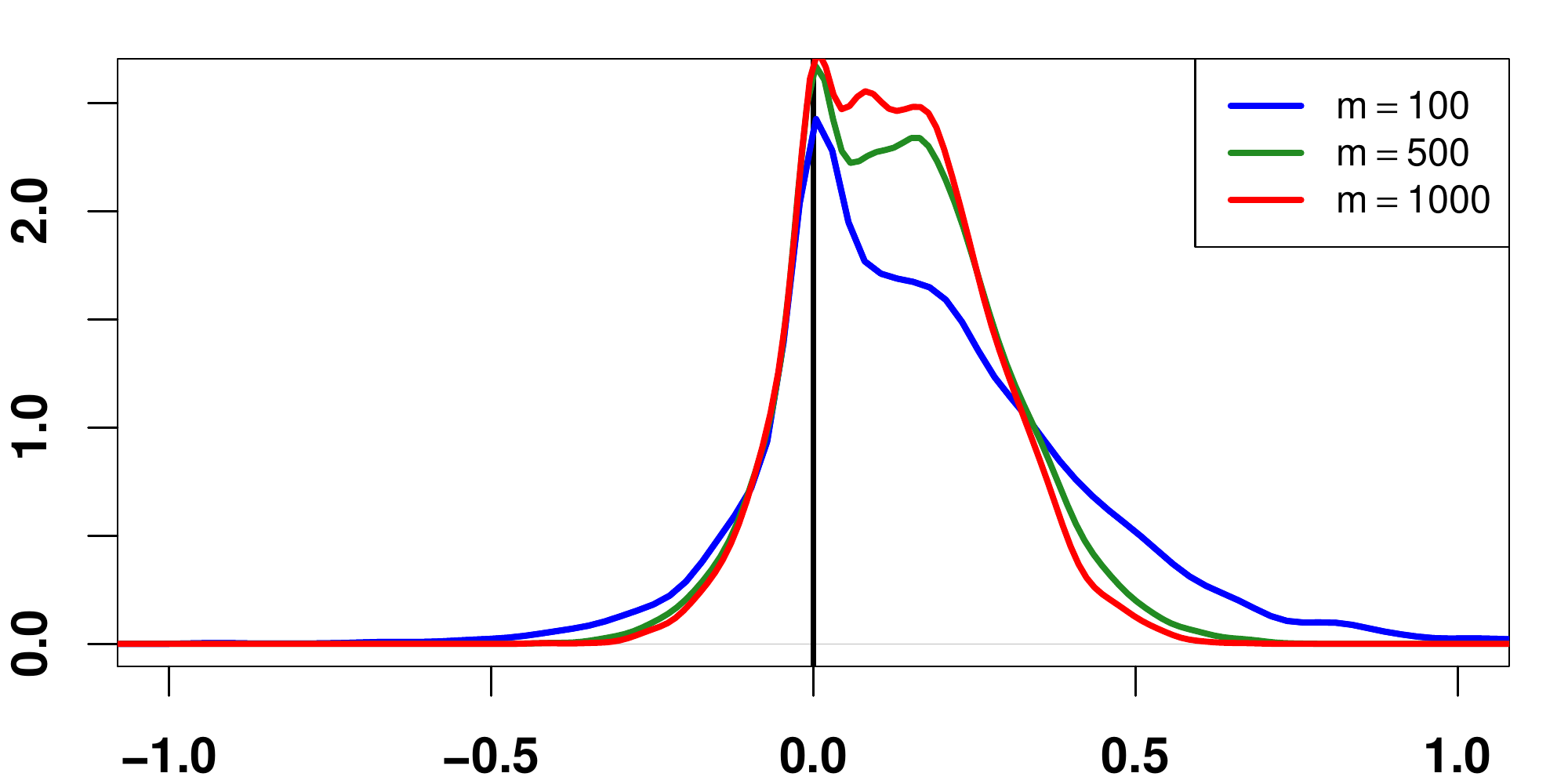}
	}\\
	\subfloat[$t(3)$, $\beta=0.25$]{
		\includegraphics[width=0.32\textwidth]{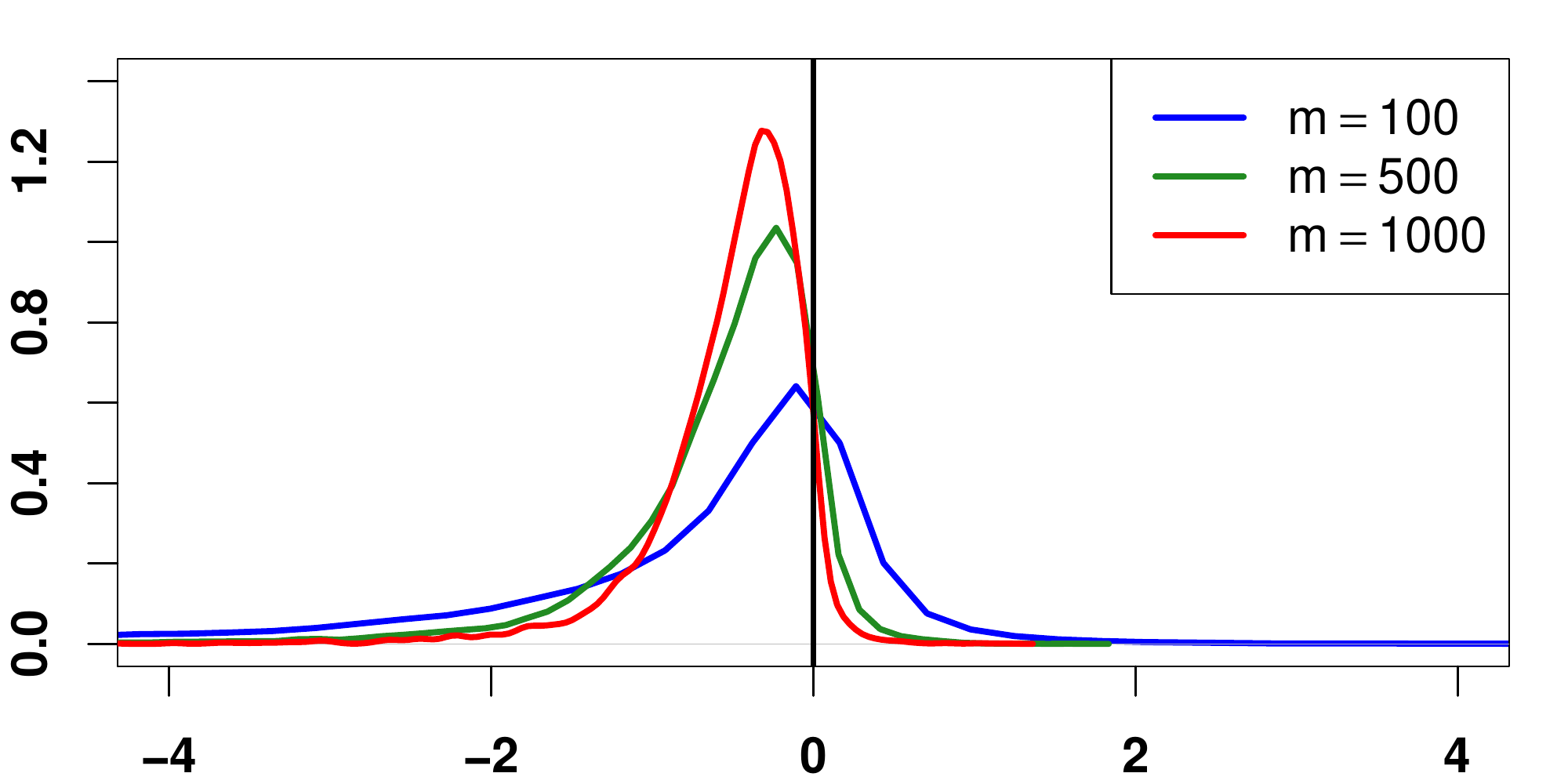}
		}
	\subfloat[$t(3)$, $\beta=0.75$]{
		\includegraphics[width=0.32\textwidth]{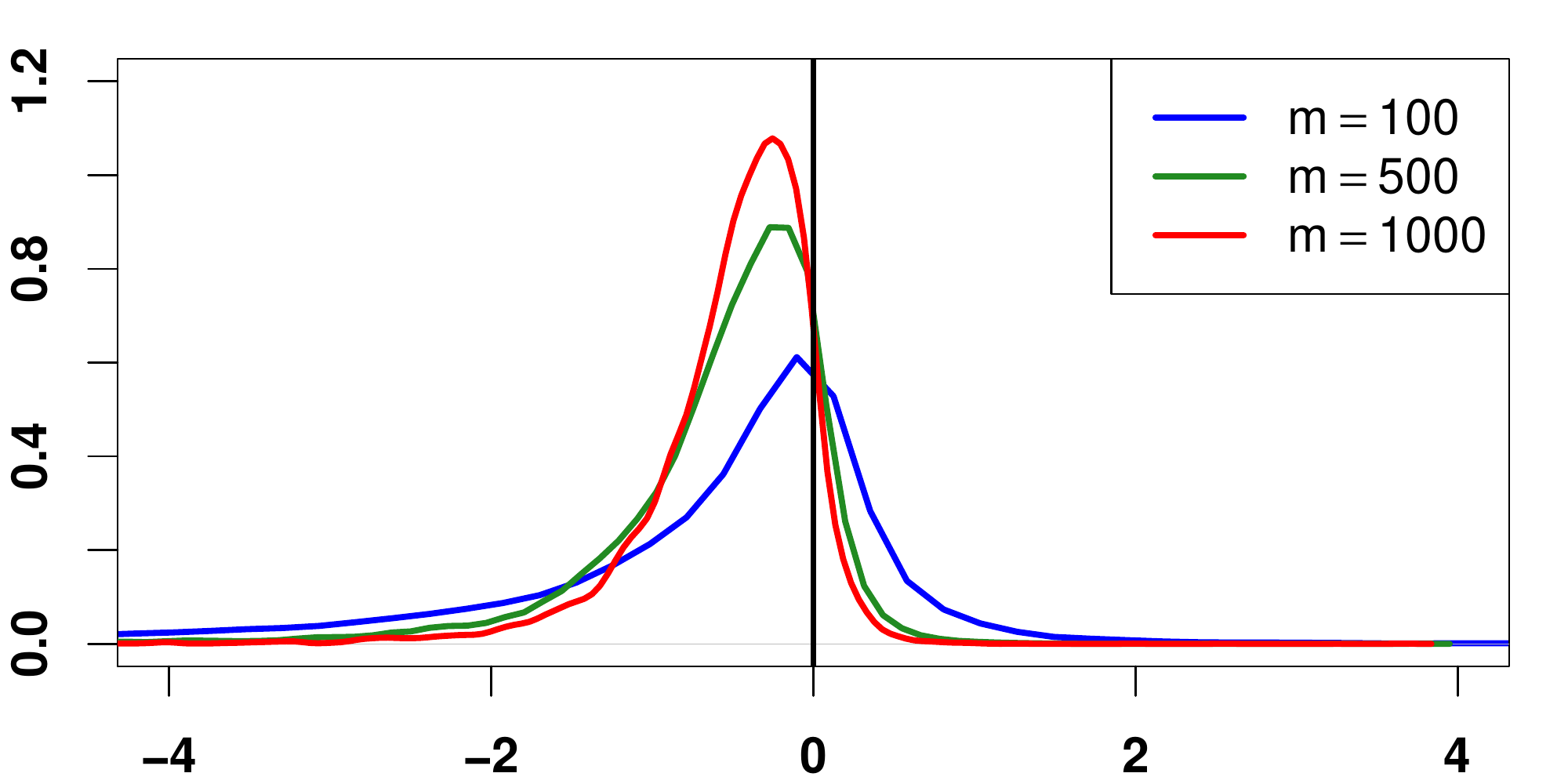}
	}
	\subfloat[$t(3)$, $\beta=1.4$]{
		\includegraphics[width=0.32\textwidth]{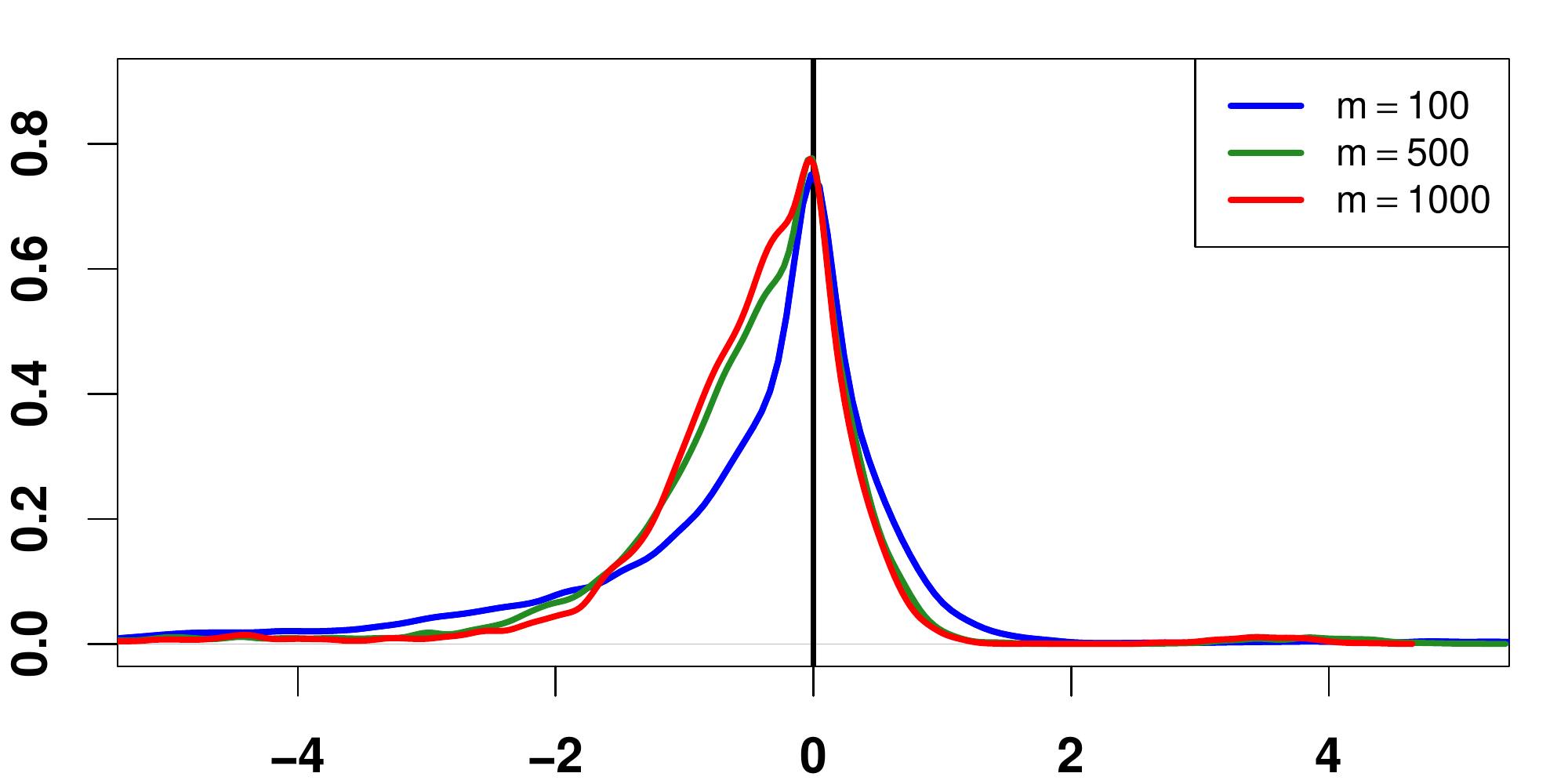}
	}

		\caption{Estimated densities 
			 of the observed delay times scaled by $\tilde{c}_m \sqrt{m}\left( 1+\frac{k^*}{m} \right)$ for AR(1) errors with parameter $a=0.2$ and normal innovations (top row) as well as $t(3)$-innvoations (bottom row). The dashed line indicates the theoretical value from Table \ref{scalediff}, while the solid line indicates where 0 lies. 		}
		\label{figure_new_WD2}
\end{figure}

\section{Conclusions}\label{section_conclusions}
In this paper we have derived the limit distribution of the stopping time of a sequential change point procedure based on $U$-statistics. While, previously,  only results for changes that occur early in the monitoring period were obtained,  we also derive such results for late changes. In the case of late changes there are two fundamental differences to the early change situation:

First, there is a positive probability of a false alarm before the change has actually occured if constant critical values are used, while this probability is asymptotically negligible for early change points. Such constant critical values are used for sequential testing as they allow to control the asymptotic false alarm rate at a fixed level. Consequently, it is not surprising that for later changes there is also a positive probability of such an alarm before the change point occurs. Early changes on the other hand occur  by definition asymptotically at the very beginning of the monitoring period such that there has been no time for a false alarm yet. 

Secondly, for late changes the stopping time depends on the historic sample, while this dependence is asymptotically negligible for early changes. By conditioning on the relevant quantities of the historic data set we first derive asymptotic results for the stopping time in all situations. From these we obtain unconditional results as well, where for late changes the expected delay time depends on the historic data set while this is not true for early changes. As a contrast the asymptotic variability of the delay time does not depend on the historic data for early or late changes.

As a side product we obtain a better approximation for early changes as compared to previous results by including an additional factor that is necessary for late changes but asymptotically negligible for early changes. Nevertheless,  taking it into account results in a much better small sample approximation  for early changes as shown in simulations, effectively removing the strong bias that has been reported in previous simulations.

Furthermore, we have derived the stopping times not only for the standard DOM monitoring procedure but also for a more robust Wilcoxon procedure.
Based on these results we  theoretically compare the stopping times of both methods revealing that the Wilcoxon procedure is significantly quicker for heavy-tailed distributions while only being somewhat slower for the normal distribution, which has also been confirmed in simulations.

In future work, the same methodology can also be applied to compare stopping times for different weight functions and different types of monitoring schemes such as Page-CUSUM~(see e.g.~\cite{Fremdt}), modified MOSUM (see e.g.~\cite{kirch2018modified}) or the monitoring scheme of \cite{gosmann2019new}. So far such a comparison was only possible for early changes and has only been done for different weight functions as well as the Page- versus classical CUSUM monitoring scheme.

%
%
%
%
\section{Proofs}\label{section_proofs}
\subsection{Proofs of Section~\ref{sec_thre}}
\begin{proof}[Proof of Theorem~\ref{as.H0}]
	With the given weight function $w(m,k)$, i.e.\ $\gamma=0$, the proof of Theorem~1 in ~\cite{seqcp} works with the choice $\tau=0$ (which is not the case for $0<\gamma<\frac 1 2$). Hence, the additional assumption given there to control the growth at the beginning of the monitoring is not necessary here. Corollary 2 in \cite{seqcp} then concludes the proof.
\end{proof}
\begin{proof}[Proof of Proposition \ref{cinf}]
	For $k\leq k^*$  we obtain by \eqref{GammaH0}, $k^*=\lfloor \lambda m^{\beta}\rfloor$,
 Assumption \ref{regass}  and Lemma 2 in \cite{seqcp} 
\begin{align*}
	&\sup_{1\leq k\leq k^*}\frac{\left|{\Gamma}(m,k)\right|}{\sqrt{m}\left(1+\frac km\right)}\leq \frac{\sqrt{\frac{k^*}{m}}}{1+\frac{k^*}{m}}\sup_{1\leq k\leq k^*}\frac{1}{\sqrt{k^*}}\left|\sum_{j=m+1}^{m+k}h_2(Y_j)\right|+ \frac{\frac{k^*}{m}}{1+\frac{k^*}{m}}|S_{1,m}|+o_P(1)\\
&=o_P(1)+\frac{\lambda}{m^{1-\beta}+\lambda}|S_{1,m}|.
\end{align*}
By \eqref{eq_lower_cm} it holds $1/c_m=O(1)$ such that the assertion follows by \eqref{cms}.
\end{proof}
\subsection{Proof of Theorem~\ref{thm:cond:version}}
As discussed in Section~\ref{sec:derivation} the proof of Theorem~\ref{thm:cond:version} is complete as soon as we  have proven Proposition~\ref{delta.lim2}, which is achieved in this section.

Some of the findings and direct consequences of the derivation of the expected asymptotic delay time and its variability of Section~\ref{sec:derivation}  are summarized in the following lemma:
\begin{lemma}\label{det.ass.gen}
	Under Assumption~\ref{change.ass}, \eqref{eq_lower_cm} and \eqref{eq_s_m} it holds
\begin{itemize}
	\item[a)] $		\frac{k^*}{N}=\frac{k^*}{a_m}(1+o(1))\to 
	\begin{cases}
		0, &\phantom{\text{if }m^{\beta -1}\,\frac{\sqrt{m}|\Delta_m|}{c_m}
	\to \;} 0,\\
	\frac{\lambda D}{1+\lambda D},&\text{if }	m^{\beta -1}\,\frac{\sqrt{m}|\Delta_m|}{c_m}\to D\in (0,\infty),\\
		1,&\phantom{\text{if }m^{\beta -1}\,\frac{\sqrt{m}|\Delta_m|}{c_m}\to \;}\infty.
	\end{cases}
$
\item[b)]$\frac{a_m-k^*}{m+k*}\to 0$, $\frac{N-k^*}{m+k*}\to 0$.
\item[c)]  $\frac{m+a_m}{m+k^*}\to 1$, $\frac{m+N}{m+k^*}\to 1$.
\item[d)]$\lim_{m\to\infty}\frac{(N-k^*)|\Delta_m|}{\sqrt{N}}\to\infty$, showing in particular, that $\sqrt{N-k^*}|\Delta_m|\to \infty$ and $N\to\infty$.
\item[e)] $\liminf_{m\to\infty} \frac{k^*\,s_{1,m}\sign(\Delta_m)}{\sqrt{m}\,(N-k^*)\,|\Delta_m|}>-\frac{1}{2}$ and $\limsup_{m\to\infty}\frac{k^*\,s_{1,m}\sign(\Delta_m)}{\sqrt{m}\,(N-k^*)\,|\Delta_m|}<\infty$.
\end{itemize}
\end{lemma}

\begin{proof}
By \eqref{defN} and \eqref{eq_bmam}, it holds
\begin{align}\label{eq_nam}
	\frac{N}{a_m}\to 1,
\end{align}
such that by 
\eqref{eq:am:approx}
it holds
\begin{align*}
	N=k^* \left( 1+\frac{c_m}{\sqrt{m}|\Delta_m|}\,\frac{m}{k^*} \right)(1+o(1)),
\end{align*}
from which (a) follows. 
By \eqref{eq:am:approx} it holds \begin{align*}
	a_m-k^*=o(k^*)+m\, O\left(\frac{c_m}{\sqrt{m}\,|\Delta_m|}\right)=o(k^*+m),
	\end{align*}
	showing the first assertion in b).
	By \eqref{eq_nam} and \eqref{ineq:amdelta} it follows
	\begin{align*}
		\frac{N-k^*}{m+k^*}=\frac{N}{a_m}\,\frac{a_m-k^*}{m+k^*} + \left( \frac{N}{a_m}-1 \right)\, \frac{k^*}{m+k^*}=o(1).
	\end{align*}
	Assertion c) is a direct consequence of b).

	Finally, by \eqref{defN}, \eqref{defam2} and \eqref{defbm2} it holds
	\begin{align}
		&\sqrt{m}\,|\Delta_m|\,(N-k^*)=\sqrt{m\,a_m}\,x+\sqrt{m}\,(a_m-k^*)\,|\Delta_m|\notag\\*
&=\sqrt{m\,a_m}\,x+k^*(s_{1,m}^*-s_{1,m})\,\sign(\Delta_m)+c_mm+a_m\left(c_m-s_{1,m}^*\sign(\Delta_m)\right)\notag\\
&=c_m(m+a_m)\left(\frac{\sqrt{m\,a_m}}{(m+a_m)\,c_m}\,x+ 1-\frac{s_{1,m}}{c_m}\frac{k^*\sign(\Delta_m)}{m+a_m}-\frac{s_{1,m}^*}{c_m} \,\frac{a_m-k^*}{m+a_m}\sign(\Delta_m)\right)\notag\\
&=c_m(m+a_m)\left( 1-\frac{s_{1,m}}{c_m}\frac{\lambda \sign(\Delta_m)}{m^{1-\beta}+\lambda}+o(1) \right),\label{eq_lem_gneu}
\end{align}
where in the last line for linear changes $c_m\to\infty$ was used and for early changes that in this case $a_m/m=o(1)$ by \eqref{eq:am:approx}.

Assertion d) follows from this and c) because $(m+N)/\sqrt{m\,N}\ge \sqrt{\max(m/N,N/m)}$. The latter converges to infinity for superlinear changes as in that case $N/m\to \infty$. It also converges to infinity for sublinear (early) changes as in that case $m/N\to 0$ by \eqref{eq_nam} and \eqref{eq:am:approx}.  For linear changes it is bounded (from above and below), but $c_m\to \infty$ so that the assertion also follows.

Furthermore, it holds by c) and \eqref{eq_lem_gneu}
\begin{align*}
	&\frac{k^*\,s_{1,m}\sign(\Delta_m)}{\sqrt{m}\,(N-k^*)\,|\Delta_m|} =\frac{\frac{s_{1,m}}{c_m}\frac{\lambda\sign(\Delta_m)}{m^{1-\beta}+\lambda}+o(1)}{1-\frac{s_{1,m}}{c_m}\frac{\lambda\sign(\Delta_m)}{m^{1-\beta}+\lambda}+o(1)},
\end{align*}
such that assertion e) follows by \eqref{eq_s_m}.
\end{proof}

\begin{lemma}\label{delta.stop2}
	Under the assumptions of Theorem~\ref{thm:cond:version} it holds for any $\delta\in(0,1)$, $z\in\mathbb{R}$ fixed,  as $m\rightarrow\infty,$
\begin{align*}
	&P\left(\sup_{1\leq l<N-k^*}\left.\frac{\left|\widetilde{\Gamma}(m,k^*+l)\right|}{\sqrt{m}\left(1+\frac{k^*+l}{m}\right)}-e_m\leq \sqrt{v_m} \,z\right|S_{1,m}=s_{1,m},S^*_{1,m}=s_{1,m}^*\right)\\
	&=P\left(\sup_{(1-\delta)(N-k^*)\leq l<N-k^*}\left.\frac{\left|\widetilde{\Gamma}(m,k^*+l)\right|}{\sqrt{m}\left(1+\frac{k^*+l}{m}\right)}-e_m\leq \sqrt{v_m}\,z\right|S_{1,m}=s_{1,m},S^*_{1,m}=s_{1,m}^*\right)+o(1).
\end{align*}
\end{lemma}
\begin{proof}
	By \eqref{eq_condprob} 
	it is sufficient to show that 
	\begin{align*}
		P&\left(\frac {m+N}{\sqrt{N}}\left(\sup_{1\leq l<(1-\delta)(N-k^*)}\frac{\left|\sum_{j=m+1}^{m+k^*}h_2(Y_j)+\sum_{j=m+k^*+1}^{m+k^*+l}h^*_{2,m}(Z_{j,m})+l\Delta_m+\frac{k^*}{\sqrt{m}}s_{1,m}+\frac{l}{\sqrt{m}}s_{1,m}^*\right|}{m+k^*+l}\notag\right.\right.\\*
		&\qquad\left.\left.-\frac{(N-k^*)|\Delta_m|+\frac{k^*}{\sqrt{m}}s_{1,m}\sign(\Delta_m)+\frac{N-k^*}{\sqrt{m}}s_{1,m}^*\sign(\Delta_m)}{m+N}\right)\leq z \sigma_{\infty} \right)\to 1.
	\end{align*}
	By Lemma~\ref{det.ass.gen} a), c) and d), \eqref{eq_s_m} and by Assumption \ref{ass.stop} b) resp.\ c)  it holds
	\begin{align*}
&	\frac {m+N}{\sqrt{N}}\sup_{1\leq l<(1-\delta)(N-k^*)}\frac{\left|\sum_{j=m+1}^{m+k^*}h_2(Y_j)+\sum_{j=m+k^*+1}^{m+k^*+l}h^*_{2,m}(Z_{j,m})+l\Delta_m+\frac{k^*}{\sqrt{m}}s_{1,m}+\frac{l}{\sqrt{m}}s_{1,m}^*\right|}{m+k^*+l}	\\
&= \frac {m+N}{\sqrt{N}}\sup_{1\leq l<(1-\delta)(N-k^*)}\frac{\left|l\Delta_m+\frac{k^*}{\sqrt{m}}s_{1,m}\right|}{m+k^*+l}	+o\left( \frac{(N-k^*)\,|\Delta_m|}{\sqrt{N}} \right)+O_P\left(1\right)\\
	&= \frac {(N-k^*)\,|\Delta_m|}{\sqrt{N}}\sup_{1\leq l<(1-\delta)(N-k^*)}\left|\frac{l}{N-k^*}+\frac{k^*\,s_{1,m}\sign(\Delta_m)}{\sqrt{m}\,(N-k^*)\,|\Delta_m|}\right|+o_P\left( \frac{(N-k^*)\,|\Delta_m|}{\sqrt{N}} \right).
	\end{align*}
Similarly,
\begin{align*}
	&\frac{1}{\sqrt{N}}\left((N-k^*)|\Delta_m|+\frac{k^*}{\sqrt{m}}s_{1,m}\sign(\Delta_m)+\frac{N-k^*}{\sqrt{m}}s_{1,m}^*\sign(\Delta_m)\right)\\
	&= \frac {(N-k^*)\,|\Delta_m|}{\sqrt{N}}\left( 1+ \frac{k^*\,s_{1,m}\sign(\Delta_m)}{\sqrt{m}\,(N-k^*)\,|\Delta_m|}\right)+o\left( \frac{(N-k^*)\,|\Delta_m|}{\sqrt{N}} \right).
\end{align*}
Furthermore,
\begin{align*}
&\sup_{1\leq l<(1-\delta)(N-k^*)}\left|\frac{l}{N-k^*}+\frac{k^*\,s_{1,m}\sign(\Delta_m)}{\sqrt{m}\,(N-k^*)\,|\Delta_m|}\right|
-\left( 1+ \frac{k^*\,s_{1,m}\sign(\Delta_m)}{\sqrt{m}\,(N-k^*)\,|\Delta_m|}\right)\\
&\le \max\left( 1-\delta +\frac{k^*\,s_{1,m}\sign(\Delta_m)}{\sqrt{m}\,(N-k^*)\,|\Delta_m|}, \frac{-k^*\,s_{1,m}\sign(\Delta_m)}{\sqrt{m}\,(N-k^*)\,|\Delta_m|} \right)
-\left( 1+ \frac{k^*\,s_{1,m}\sign(\Delta_m)}{\sqrt{m}\,(N-k^*)\,|\Delta_m|}\right)\\
&\le - \min\left( \delta, 1+2\,\frac{k^*\,s_{1,m}\sign(\Delta_m)}{\sqrt{m}\,(N-k^*)\,|\Delta_m|} \right),
\end{align*}
where the latter minimum is positive and bounded away from zero by Lemma~\ref{det.ass.gen} e).
Lemma~\ref{det.ass.gen} d) completes the proof.\end{proof}

\begin{proposition}\label{delta.lim2}
Let  $s_{1,m}, s_{1,m}^*\in\mathbb{R}$ with $P(S_{1,m}=s_{1,m},S^*_{1,m}=s_{1,m}^*)>0$. Under the assumptions of Theorem \ref{thm:cond} it holds for all $z\in\mathbb{R}$
\begin{align*}
\lim_{m\rightarrow\infty}P\left(v_m^{-1/2}\left(\sup_{k^*<k<N}\left.\frac{\left|\tilde{\Gamma}(m,k)\right|}{\sqrt{m}\left(1+\frac{k}{m}\right)}-e_m\right)
\leq z\right|S_{1,m}=s_{1,m},S^*_{1,m}=s_{1,m}^*\right)=\Phi\left(z\right)
\end{align*}
with $e_m$ and $v_m$ as in \eqref{eq:em} and \eqref{eq:vm}.
\end{proposition}
\begin{proof}
	By  Assumption~\ref{ass.stop} (b) resp.\ (c) it holds for $(1-\delta)(N-k^*)<l\le N-k^*$ uniformly in $0\le \delta\le \frac{1}{2}$
\begin{align*}
	&\sign(\Delta_m)\left(\sum_{j=m+1}^{m+k^*}h_2(Y_j)+\sum_{j=m+k^*+1}^{m+k^*+l}h^*_{2,m}(Z_{j,m})+l\Delta_m+\frac{k^*}{\sqrt{m}}s_{1,m}+\frac{l}{\sqrt{m}}s_{1,m}^*\right)\\
	&\ge |\Delta_m|(N-k^*)\left( 1-\delta + \frac{k^*s_{1,m}\sign(\Delta_m)}{\sqrt{m}|\Delta_m|(N-k^*)}+ O_P\left(\frac{1}{|\Delta_m|\sqrt{(N-k^*)}}\right)+O\left( \frac{|s_{1,m}^*|}{\sqrt{m}|\Delta_m|} \right)\right)>0,
\end{align*}
where the positivity holds for $m$ large enough by \eqref{eq_s_m} and Lemma~\ref{det.ass.gen} d) and e).
Consequently, uniformly in $0\le \delta\le \frac 1 2$, it holds for $m$ large enough uniformly in $l$
\begin{align*}
	&\sup_{(1-\delta)(N-k^*)<l\le N-k^*}\frac{\left|\sum_{j=m+1}^{m+k^*}h_2(Y_j)+\sum_{j=m+k^*+1}^{m+k^*+l}h^*_{2,m}(Z_{j,m})+l\Delta_m+\frac{k^*}{\sqrt{m}}s_{1,m}+\frac{l}{\sqrt{m}}s_{1,m}^*\right|}{m+k^*+l}\\
	&=\sup_{(1-\delta)(N-k^*)<l\le N-k^*}\sign(\Delta_m)\frac{\left(\sum_{j=m+1}^{m+k^*}h_2(Y_j)+\sum_{j=m+k^*+1}^{m+k^*+l}h^*_{2,m}(Z_{j,m})+l\Delta_m+\frac{k^*}{\sqrt{m}}s_{1,m}+\frac{l}{\sqrt{m}}s_{1,m}^*\right)}{m+k^*+l}\end{align*}
On the one hand, setting $l=\lfloor N-k^*\rfloor$ and noting that $(m+N)/(m+\lfloor N\rfloor)=1+O(1/(m+N))$, $s_{1,m}/m=o(1)$ as well as $s_{1,m}^*/\sqrt{m}=o(1)$  this implies by Lemma~\ref{det.ass.gen} (a) and (b)
\begin{align*}
	&\frac{m+N}{\sqrt{N}}\,\sup_{(1-\delta)(N-k^*)<l\le N-k^*}\frac{\left|\sum_{j=m+1}^{m+k^*}h_2(Y_j)+\sum_{j=m+k^*+1}^{m+k^*+l}h^*_{2,m}(Z_{j,m})+l\Delta_m+\frac{k^*}{\sqrt{m}}s_{1,m}+\frac{l}{\sqrt{m}}s_{1,m}^*\right|}{m+k^*+l}\\
	&\ge\frac{m+N}{m+\lfloor N\rfloor}\left( \sign(\Delta_m)\frac{1}{\sqrt{N}}\sum_{j=m+1}^{m+k^*}h_2(Y_j)+\sign(\Delta_m)\frac{1}{\sqrt{N}}\sum_{j=m+k^*+1}^{m+N}h^*_{2,m}(Z_{j,m}) \right. \\*
	&\left.
	\phantom{\frac{m+N}{m+\lfloor N\rfloor}}\qquad+\frac{\lfloor N-k^*\rfloor |\Delta_m|}{\sqrt{N}}+\frac{k^* s_{1,m}\sign(\Delta_m)}{\sqrt{N\,m}}+\frac{\lfloor N-k^*\rfloor s_{1,m}^*\sign(\Delta_m)}{\sqrt{N\,m}}
 \right)\\
 &=  \sign(\Delta_m)\frac{1}{\sqrt{N}}\sum_{j=m+1}^{m+k^*}h_2(Y_j)+\sign(\Delta_m)\frac{1}{\sqrt{N}}\sum_{j=m+k^*+1}^{m+N}h^*_{2,m}(Z_{j,m})  \\*
	&
\qquad+ \frac{(N-k^*) |\Delta_m|}{\sqrt{N}}+\frac{k^* s_{1,m}\sign(\Delta_m)}{\sqrt{N\,m}}+\frac{ (N-k^*) s_{1,m}^*\sign(\Delta_m)}{\sqrt{N\,m}}
+o_P(1).
\end{align*}
On the other hand, by Lemma~\ref{det.ass.gen} b) and c) uniformly in $0\le \delta<\frac 1 2$ for $m$ large enough
\begin{align*}
	&\frac{m+N}{\sqrt{N}}\,\sup_{(1-\delta)(N-k^*)<l\le N-k^*}\frac{\left|\sum_{j=m+1}^{m+k^*}h_2(Y_j)+\sum_{j=m+k^*+1}^{m+k^*+l}h^*_{2,m}(Z_{j,m})+l\Delta_m+\frac{k^*}{\sqrt{m}}s_{1,m}+\frac{l}{\sqrt{m}}s_{1,m}^*\right|}{m+k^*+l}\\
&\le \sign(\Delta_m)\frac{1}{\sqrt{N}}\sum_{j=m+1}^{m+k^*}h_2(Y_j)\, (1+o(1))\\*
&\quad+ 
\sup_{(1-\delta)(N-k^*)<l\le N-k^*}\left(1-\frac{N-k^*-l}{m+k^*+l}\right)\sign(\Delta_m)\,
\frac{1}{\sqrt{N}}\,
	\sum_{j=m+k^*+1}^{m+k^*+l}h^*_{2,m}(Z_{j,m})\\*
	&\quad+\frac{m+N}{\sqrt{N}}\,\sup_{(1-\delta)(N-k^*)<l\le N-k^*}\frac{ l|\Delta_m|+\frac{k^*}{\sqrt{m}}s_{1,m}\sign(\Delta_m)+\frac{l}{\sqrt{m}}s_{1,m}^*\sign(\Delta_m) }{m+k^*+l}\\*
	&=\sign(\Delta_m)\frac{1}{\sqrt{N}}\sum_{j=m+1}^{m+k^*}h_2(Y_j)
	+\sup_{(1-\delta)(N-k^*)<l\le N-k^*}\sign(\Delta_m)\,
\frac{1}{\sqrt{N}}\,
	\sum_{j=m+k^*+1}^{m+k^*+l}h^*_{2,m}(Z_{j,m})\\*
	&\qquad + \frac{(N-k^*) |\Delta_m|}{\sqrt{N}}+\frac{k^* s_{1,m}\sign(\Delta_m)}{\sqrt{N\,m}}+\frac{ (N-k^*) s_{1,m}^*\sign(\Delta_m)}{\sqrt{N\,m}}
+o_P(1).
\end{align*}
In the last equation we used the fact, that the 
last supremum is taken in $l=\lfloor N-k^*\rfloor$ for $m$ large enough, because by Lemma~\ref{det.ass.gen} b) and (e) in combination with
\eqref{eq_s_m}  
 the following representation holds
\begin{align*}
	&\frac{ l|\Delta_m|+\frac{k^*}{\sqrt{m}}s_{1,m}\sign(\Delta_m)+\frac{l}{\sqrt{m}}s_{1,m}^*\sign(\Delta_m) }{m+k^*+l}\\
	&= |\Delta_m|\left( 1+\frac{s_{1,m}^*\sign(\Delta_m)}{\sqrt{m}|\Delta_m|} \right)\,\frac{l+(N-k^*)\frac{k^*s_{1,m}\sign(\Delta_m)}{\sqrt{m}(N-k^*)|\Delta_m|}\,\left( 1+\frac{s_{1,m}^*\sign(\Delta_m)}{\sqrt{m}|\Delta_m|} \right)^{-1}}{m+k^*+l}\\
&	=|\Delta_m| (1+o(1))\,\frac{l+ (m+k^*) o(1)}{m+k^*+l},
\end{align*}
where the last term is increasing in $l$ for $m$ large enough.

Putting the above together shows that
\begin{align*}
&\frac{m+N}{\sqrt{N}}\,\sup_{(1-\delta)(N-k^*)<l\le N-k^*}\frac{\left|\sum_{j=m+1}^{m+k^*}h_2(Y_j)+\sum_{j=m+k^*+1}^{m+k^*+l}h^*_{2,m}(Z_{j,m})+l\Delta_m+\frac{k^*}{\sqrt{m}}s_{1,m}+\frac{l}{\sqrt{m}}s_{1,m}^*\right|}{m+k^*+l}\\
&\qquad -
\frac{(N-k^*)|\Delta_m|+\frac{k^*}{\sqrt{m}}s_{1,m}\sign(\Delta_m)+\frac{N-k^*}{\sqrt{m}}s_{1,m}^*\sign(\Delta_m)}{\sqrt{N}}\\
&=\sign(\Delta_m)\left( \sqrt{\frac{k^*}{N}}\,\frac{1}{\sqrt{k^*}}\sum_{j=m+1}^{m+k^*}h_2(Y_j)
+\sqrt{\frac{N-k^*}{N}}\frac{1}{\sqrt{N-k^*}}\sum_{j=m+k^*+1}^{m+N}h^*_{2,m}(Z_{j,m}) \right)+o_P(1)\\
&+\sqrt{\frac{N-k^*}{N}}\,O\left( 
\sup_{(1-\delta)(N-k^*)<l\le N-k^*}\left|\frac{1}{\sqrt{N-k^*}}\,
	\sum_{j=m+k^*+1}^{m+k^*+l}h^*_{2,m}(Z_{j,m})
-\frac{1}{\sqrt{N-k^*}}\sum_{j=m+k^*+1}^{m+N}h^*_{2,m}(Z_{j,m})\right|
\right).
\end{align*}
As soon as $m^{\beta -1}\,\frac{\sqrt{m}|\Delta_m|}{c_m}\to \infty$, the factor in front of the last term converges to zero by Lemma~\ref{det.ass.gen}~a), while by Assumption~\ref{ass.stop} (c) the stochastic part is bounded in probability, such that the full term is $o_P(1)$.
If that is not the case, then by Assumption~\ref{ass.stop} (b)
\begin{align*}
&\sup_{(1-\delta)(N-k^*)<l\le N-k^*}\left|\frac{1}{\sqrt{N-k^*}}\,
	\sum_{j=m+k^*+1}^{m+k^*+l}h^*_{2,m}(Z_{j,m})
-\frac{1}{\sqrt{N-k^*}}\sum_{j=m+k^*+1}^{m+N}h^*_{2,m}(Z_{j,m})\right|\\
&
= \sup_{1-\delta<s<1}|W^*(s)-W^*(1)|+o_P(1),
\end{align*}
where by the almost sure continuity of a Wiener process, the last term converges to 0 for $\delta\to 0$. Because all the other $o$-terms were uniformly in $\delta$, this gives the result by Lemma~\ref{delta.stop2}, Assumption~\ref{ass.stop} (b) or (c) in addition to Lemma~\ref{det.ass.gen} a).\end{proof}

\subsection{Proof of Theorem~\ref{thm:cond}}
We are now ready to prove Theorem~\ref{thm:cond}.

\begin{proof}[Proof of Theorem \ref{thm:cond}]
	By Lemma~\ref{det.ass.gen} (c)  and Assumption~\ref{regass}(i) it holds
	\begin{align*}
	&\frac{m+N}{\sqrt{N}}\sup_{k^*< k\leq N}\frac{\left|\frac{1}{m}\sum_{i=1}^m\sum_{j=m+1}^{m+k^*}r(Y_i,Y_j)\right|}{m+k}
	=O_P\left( \sqrt{\frac{u(m)}{m^2}} \right)=o_P(1).
\end{align*}
With Assumption~\ref{ass.stop} (a)(i), Lemma~\ref{det.ass.gen} (c) and Theorem 3 in	\cite{Momineq} we get
\begin{align*}
	&\frac{m+N}{\sqrt{N}}\sup_{k^*< k\leq N}\frac{\left|\frac{1}{m}
	\sum_{i=1}^m\sum_{j=m+k^*+1}^{m+k}r^*_m(Y_i,Z_{j,m})
\right|}{m+k}
	=O_P\left( \sqrt{\frac{u(m)\log^2(m)}{m^2}}  \right)=o_P(1),
\end{align*}
showing that for any $\epsilon>0$ it holds for $m\to\infty$
\begin{align*}
	\E\left(P\left( \frac{m+N}{\sqrt{N}}\sup_{k^*< k\leq N}\frac{\left|R(m,k)\right|}{m+k}>\epsilon\,\Big|\,S_{1,m},S_{1,m}^*\right) \right)=
	P\left( \frac{m+N}{\sqrt{N}}\sup_{k^*< k\leq N}\frac{\left|R(m,k)\right|}{m+k}>\epsilon\right)
	\to 0.
\end{align*}
Hence
\begin{align*}
	P\left( \frac{m+N}{\sqrt{N}}\sup_{k^*< k\leq N}\frac{\left|R(m,k)\right|}{m+k}>\epsilon\,\Big|\,S_{1,m},S_{1,m}^*\right)=o_P(1).
\end{align*}
Because \eqref{eq_s_m} is also fulfilled in a $P$-stochastic sense  for $S_{1,m}$ and $S_{1,m}^*$ by Assumption~\ref{ass_neu_stoch}, the assertion follows by an application of the subsequence principle and Theorem~\ref{thm:cond:version}.
\end{proof}

\begin{proof}[Proof of Corollary \ref{cor:cond}]
	The corollary follows directly from  Theorem~\ref{thm:cond:version} because \eqref{eq_s_m} is fulfilled almost surely for $S_{1,m}$ and $S_{1,m}^*$ by Assumption~\ref{ass_as_S}.
\end{proof}

\subsection{Proofs of Section~\ref{sec:imp}}\label{proof_sec_imp}
In this section we prove Theorem~\ref{thm:aprime}, a proof of Theorem~\ref{thm:delay:early} can be found in \cite{diss}, Theorem 5.3.
\begin{proof}[Proof of Theorem~\ref{thm:aprime}]
	The arguments leading to \eqref{eq:am:approx} show that
	\begin{align}\label{eq_am3}
		a_m(S_{1,m},S_{1,m}^*)=\left(k^*+\frac{c_m\,\sqrt{m}}{|\Delta_m|}\right)\, \left( 1+o_P(1) \right).
	\end{align}
	The same assertion holds for $a_m(S_{1,m},0)$ as well as $a_m(0,0)$, where in the latter case $o_P(1)$ can be replaced by $o(1)$. Additionally, the assertions of Lemma~\ref{det.ass.gen} also hold with $o(1)$ replaced by $o_P(1)$ where appropriate.
Consequently,
\begin{align}\label{eq_new_ck2}
		\frac{b^2_m(S_{1,m},S_{1,m}^*)}{b^2_m(0,0)}={\frac{a_m(S_{1,m},S_{1,m}^*)}{a_m(0,0)}}\stackrel{P}{\rightarrow} 1.
	\end{align}
and assertion a) follows with Corollary~\ref{cor_uncond}. Furthermore,
\begin{align}\label{eq_new_ck}
	\frac{a_m(0,0)\,|\Delta_m|^2}{c_m^2}\ge \frac{|\Delta_m|\sqrt{m}}{c_m}\,(1+o(1))\to \infty.
\end{align}
By definition it holds
\begin{align*}
	&a_m(S_{1,m},S_{1,m}^*)-k^*\\
&	=	\frac{c_m\,a_m(S_{1,m},S_{1,m}^*)}{\sqrt{m}|\Delta_m|}-\frac{a_m(S_{1,m},S_{1,m}^*)}{\sqrt{m}|\Delta_m|}\,S_{1,m}^*\sign(\Delta_m)
	+\frac{k^*}{\sqrt{m}|\Delta_m|}(S_{1,m}^*-S_{1,m})\sign(\Delta_m)+\frac{c_m\sqrt{m}}{|\Delta_m|}\\
	&=\sqrt{m\,a_m(0,0)}\;\left(O_P(1)\,	\frac{c_m\,\sqrt{a_m(0,0)}}{m|\Delta_m|}
	+O_P\left( \frac{\sqrt{k^*}}{m\,|\Delta_m|} \right)
	+ o(1)
 \right), 
\end{align*}
where the last line follows from
 \eqref{eq_new_ck2}, \eqref{eq_new_ck} and by $k^*/a_m(0,0)=O(1)$ as in Lemma~\ref{det.ass.gen} a).
 
 For late changes it also holds  $a_m(0,0)/k^*=O(1)$  as in Lemma~\ref{det.ass.gen} a), for early changes that $a_m(0,0)=o_P(m)$ as can be seen e.g.\ by \eqref{eq_am3}. Consequently,
 \begin{align*}
	&\frac{a_m(S_{1,m},S_{1,m}^*)-k^*}{\sqrt{a_m(0,0)\,m}}
	=O_P(1)\,\frac{c_m}{\sqrt{m}\,|\Delta_m|}\,\max(1,m^{\frac{\beta-1}{2}})+o(1).
\end{align*}
Finally, by some calculations this yields
\begin{align*}
	&	\frac{|a_m(S_{1,m},0)-a_m(S_{1,m},S_{1,m}^*)|}{b_m(0,0)}\left( 1-\frac{c_m}{\sqrt{m}|\Delta_m|} \right)=\left(a_m(S_{1,m},S_{1,m}^*)-k^*  \right)\,\frac{|S_{1,m}^*|}{\sqrt{a_m(0,0)\,m}}\\
	&=O_P(1)\,\frac{c_m}{\sqrt{m}\,|\Delta_m|}\,\max(1,m^{\frac{\beta-1}{2}})+o(1),
\end{align*}
proving b).
Similarly,
\begin{align*}
	(a_m(0,0)-a_m(S_{1,m},0))\,\left( 1-\frac{c_m}{\sqrt{m}|\Delta_m|} \right)=\frac{k^*}{\sqrt{m}|\Delta_m|}\,S_{1,m}\sign(\Delta_m),
\end{align*}
such that
\begin{align*}
	\frac{a_m(0,0)-a_m(S_{1,m},0)}{b_m(0,0)}=O_P(1)\,\frac{k^*}{\sqrt{a_m(0,0)\,m}},
\end{align*}
where the last term converges to zero for sublinear changes only as in Lemma~\ref{det.ass.gen} a), completing the proof of c).
\end{proof}
\begin{proof}[Proof of Remark~\ref{rem_1_imp}]
	\begin{align*}
	&	
	\frac{a_m(0,0)-\left(k^*+\frac{c_m\sqrt{m}}{|\Delta_m|}\right)}{b_m(0,0)}
	=\frac{a_m(0,0)}{b_m(0,0)}  \frac{c_m}{\sqrt{m}|\Delta_m| }	= c_m\,\sqrt{\frac{a_m(0,0)}{m}},
	\end{align*}
	which converges to zero in the sublinear case only (see e.g.~\eqref{eq_am3} and Lemma~\ref{det.ass.gen} a).
\end{proof}

\begin{proof}[Proof of Remark~\ref{rem:gamma:early}]
If $\lim_{m\to\infty}(k^*/m)^{1-\gamma_2}\,\sqrt{m}|\Delta_m|=0$, then by (5.15) in \cite{diss} it holds
	\begin{align*}
		a_m(\gamma_2)=(1+o(1))\,\left( m\,c_{\gamma_2}^{\frac{1}{1-\gamma_2}}\,\left(\sqrt{m}|\Delta_m|\right)^{\frac{1}{\gamma_2-1}} \right).
	\end{align*}
Then either an analogous assertion holds for $a_m(\gamma_1)$ with $\gamma_1>\gamma_2$ such that 
$$	\frac{a_m(\gamma_1)}{a_m(\gamma_2)}=O(1) (\sqrt{m}|\Delta_m|)^{1/(\gamma_1-1)-1/(\gamma_2-1)}\to 0,$$ 
or by (5.15) in \cite{diss} it holds $a_m(\gamma_1)=O(k^*)$, such that
$$
\frac{a_m(\gamma_1)}{a_m(\gamma_2)}=O(1)\left( \left(\frac{k^*}{m}\right)^{1-\gamma_2} \sqrt{m} |\Delta_m|\right)^{\frac{1}{1-\gamma_2}}\to 0,
$$
such that $a_m(\gamma_1)<a_m(\gamma_2)$ for $m$ large enough.  

Consider now the case, where $\liminf_{m\to\infty}(k^*/m)^{1-\gamma_2}\,\sqrt{m}|\Delta_m|>0$. Then an analogous assertion also holds with $\gamma_1>\gamma_2$ (noting that in the sublinear case $k^*/m<1$ eventually, so that the expression is increasing in $\gamma$), such that  by (5.15) in \cite{diss} we get that $a_m(\gamma_1)/a_m(\gamma_2)$ is bounded.
Furthermore, by definition
\begin{align*}
	a_m(\gamma)=k^*+\frac{c_{\gamma} m^{1/2}}{|\Delta_m|}\,\left(\frac{a_m(\gamma)}{m}\right)^{\gamma},
\end{align*}
such that \begin{align*}
	\frac{a_m(\gamma_1)-k^*}{a_m(\gamma_2)-k^*}=\frac{c_{\gamma_1}}{c_{\gamma_2}}\,\left(\frac{a_m(\gamma_1)}{a_m(\gamma_2)}\right)^{\gamma_2}\, \left( \frac{a_m(\gamma_1)}{m} \right)^{\gamma_1-\gamma_2}\to 0
\end{align*}
 by Lemma 5.4 (i) in \cite{diss}, completing the proof in this case.
\end{proof}

\subsection*{Acknowledgements}
This work has been supported by the Research Training Group "Mathematical Complexity Reduction" (314838170, GRK 2297 MathCoRe) and by the Collaborative Research Center "Statistical modeling of nonlinear dynamic processes" (SFB 823, Teilprojekt C1) of the Deutsche Forschungsgemeinschaft (DFG, German Research Foundation).

\bibliography{BIB}

\end{document}